\documentclass[12pt]{amsart}
\usepackage[lite]{amsrefs} 
\usepackage{amssymb,amsmath,enumitem}
\usepackage{graphicx}
\usepackage{color}

\DeclareMathOperator{\capa}{cap}
\newcommand{\defeq}{\mathrel{\mathop:}=}
\newcommand{\K}{{\mathbf K}}
\newcommand{\G}{{\mathbf G}}
\newcommand{\R}{{\mathbb R}}
\newcommand{\M}{\mathcal{M}^{+}}
\numberwithin{equation}{section}
\newtheorem{theorem}{Theorem}[section]
\newtheorem{lemma}[theorem]{Lemma}

\newtheorem{cor}[theorem]{Corollary}

\allowdisplaybreaks

\begin{document}
\title[Global pointwise estimates] {Global pointwise estimates of positive solutions to  sublinear equations}
\author{ Igor E. Verbitsky}
\address{Department of Mathematics, University of Missouri, Columbia, MO  65211, USA}
\email{verbitskyi@missouri.edu}

\begin{abstract}
	We give  bilateral pointwise estimates for positive 
	solutions $u$ to the sublinear integral equation
		\[ u = \mathbf{G}(\sigma u^q) + f  \quad \textrm{in} \,\, \Omega,\]
for	$0 < q < 1$, where  $\sigma\ge 0$ is a measurable function,  
or a Radon measure, $ f \ge 0$, and $\mathbf{G}$ is the  integral operator associated with a positive 
 kernel $G$ on $\Omega\times\Omega$. Our main results, which include the existence criteria and uniqueness of solutions,  
hold for quasi-metric, or quasi-metrically modifiable kernels $G$.

	As a consequence, we obtain  bilateral estimates, along with the existence and uniqueness, for  positive solutions $u$, possibly unbounded, to   sublinear elliptic equations involving the fractional  Laplacian,  
		\[ (-\Delta)^{\frac{\alpha}{2}} u = \sigma  u^q + \mu \quad \textrm{in} \,\, \Omega, \qquad u=0 \, \, \textrm{in} \,\, \Omega^c, \]
 where $0<q<1$, and $\mu, \sigma \ge 0$ 
are  measurable functions, or Radon measures, on a bounded uniform domain $\Omega \subset \R^n$ 
for  $0 < \alpha \le 2$, 
 or  on the entire space $\R^n$, a ball or half-space, for $0 < \alpha <n$. 
	\end{abstract}

\maketitle
\tableofcontents

\section{Introduction}

	Let $\Omega$ be a locally compact, Hausdorff space with countable base.
	We denote by $\mathcal{M}^+(\Omega)$  the class of  nonnegative Radon measures (locally finite)  
in  $\Omega$. Let  $G\colon \Omega\times \Omega\to (0, +\infty]$ 
	be a  lower semicontinuous function (kernel)  on $\Omega\times \Omega$.  The $G$-potential of $\sigma \in \mathcal{M}^+(\Omega)$ is defined by 
	$$\G \sigma(x) =\int_{\Omega} G(x, y) \, d \sigma(y), \quad x \in \Omega. $$ 
	If $d \nu = f \, d \sigma$, where  $f \in L^1_{{\rm loc}} (\Omega, \sigma)$,  we write 
	\[
	\G \nu = \G(f \, d \sigma)=\G^\sigma f. 
	\]

	For $\sigma \in \M(\Omega)$ ($\sigma\not=0$), we study pointwise behavior of  solutions to the sublinear integral equation 
	\begin{equation}\label{sublin-eq-f} 
			u =  \G(u^q d \sigma) + f, \quad  u \ge 0 \quad \textrm{in} \, \, \Omega,  
				\end{equation}
	where $0 < q < 1$ and $f \ge 0$ is a Borel measurable function, under certain assumptions on  $G$.

	Our main goal is to obtain bilateral pointwise estimates of all solutions  
	$u$ to \eqref{sublin-eq-f} for \textit{quasi-metric} kernels $G$ (see the definition below).

	The linear case $q=1$  
was treated earlier in \cite{FNV} for non-homogeneous equations ($f \not=0$) and quasi-metric kernels $G$. The ``smallness''  
 of the operator  norm,
   \begin{equation}\label{small}
\Vert \G^\sigma\Vert_{L^2(\Omega, \sigma) \rightarrow L^2(\Omega, \sigma)}<1,  
  	\end{equation}
plays a crucial role. Under this assumption, bilateral pointwise estimates of the 
minimal solution $u=(I-\G^\sigma)^{-1} f$ were obtained in \cite{FNV}. 
	
	The sublinear case $0<q<1$ differs from the linear one in many ways. 	
	Related weighted norm inequalities are of $(1, q)$-type, instead of  
	the $(2, 2)$-type used in 	\eqref{small}. In contrast to the case $q=1$, there is no ``smallness'' assumption on $\sigma$. For $0<q<1$, we first obtain bilateral estimates of nontrivial solutions to  the homogeneous equation  ($f=0$), which are then invoked in  the study of non-homogeneous equations. 	
	To estimate all solutions pointwise, we use nonlinear potentials  
	$\K \sigma$ intrinsic 
	to sublinear problems, along with linear potentials $\G \sigma$. 
	
	For the superlinear case $q>1$, we refer to \cite{KV},  where quasi-metric kernels 
	were introduced for the first time in the 
	framework of nonlinear equations. Bilateral pointwise estimates of the (minimal) solution were obtained  for $f\not=0$ 
	under the ``smallness'' assumption  
	 \begin{equation}\label{small-q}
  \Vert \G^\sigma\Vert_{L^{q'}(\Omega, \sigma) \rightarrow L^{q'}(\Omega, \nu)}\le c,  
  	\end{equation}
  where $\frac{1}{q} + \frac{1}{q'}=1$, $d \nu=f^q d \sigma$, and $c$ depends on $q, G$. In \cite{GV},  more precise,  but one-sided estimates  of solutions   were given for all  $q\not=0$.

In this paper, we make use of some elements of potential theory outlined in \cite{Brelot}, \cite{Fug}, along with a sublinear version of Schur's lemma obtained in \cite{QV2} for kernels $G$ which are quasi-symmetric (QS) and satisfy a  weak maximum principle  (WMP); see  Sec. \ref{background} below.
In particular, the kernels $G$ considered below are  assumed to be lower semicontinuous    
	 on $\Omega\times \Omega$. 
	 This makes it possible to invoke the notion of \textit{capacity} $\textrm{cap} (\cdot)$ associated 
	 with $G$.

	 The restriction 
	  $G(x,y)>0$ for all $x, y\in \Omega$, rather than $G(x,y)\ge 0$  used in \cite{Fug}, \cite{QV2}, is introduced mainly for the sake of simplicity. It can be relaxed in many instances, and often replaced with the condition $G(x, x)>0$ for all $x\in \Omega$.

	\smallskip
	\noindent
	\textbf{Definition.} A kernel $G$ on 
	$\Omega\times \Omega$ 
	is said to be quasi-metric if $G$ is symmetric, i.e., $G(x, y)=G(y, x)$ for all $x, y \in \Omega$, and 
	 $d(x, y)\defeq \frac{1}{G(x, y)}$ satisfies  the quasi-triangle inequality   
		\begin{equation}\label{quasitr} 
			d(x, y) \le \kappa [d(x, z) + d(z, y)], \qquad \forall \, x, y, z \in \Omega,
		\end{equation}
		with quasi-metric constant $\kappa>0$. 
		\smallskip
	
	\smallskip
	\noindent
	\textbf{Remark.} 	Without loss of generality, we assume in the definition above that $d$ is nontrivial:  $d(x, y)\not=0$ for some $x, y\in \Omega$; in this case, $\kappa \ge \frac{1}{2}$. 
	\smallskip
	
Quasi-metric kernels have numerous applications in Analysis and PDE, including  weighted norm inequalities,   Schr\"odinger operators and 3-G inequalities, spectral theory, semilinear elliptic problems on $\R^n$ and 
 complete,   non-compact Riemannian manifolds, 
   etc. 
		(see, for instance, \cite{An}, \cite{FNV}, \cite{GSV}, \cite{H}, \cite{HN}, \cite{KV}, \cite{NV}, \cite{QV2}). 

More generally, we consider equations \eqref{sublin-eq-f} with \textit{quasi-metrically modifiable} kernels $G$ 	discussed  below (see also \cite{FV1}, \cite{FV2}, \cite{GV}, \cite{QV2}).  Such results are applicable to sublinear elliptic problems in a domain  $\Omega \subseteq \R^n$ (a non-empty connected open set) 
under certain restrictions on $\Omega$.  We consider equations    involving the fractional Laplacian $(-\Delta)^{\frac{\alpha}{2}}$, 
		\begin{equation}
			\label{frac_lap_eqn}
			\begin{cases}
				(-\Delta)^{\frac{\alpha}{2}} u =\sigma u^q + \mu, \,   & u>0 \, \text{ in } \Omega, \\
				u = 0 & \text{ in $\Omega^c$},
			\end{cases} 
		\end{equation}
		where $0<q<1$, $0<\alpha<n$, and $\mu, \sigma \in \mathcal{M}^+(\Omega)$. 
		
		Sublinear problems of this type have been  extensively studied, especially in the 
		classical case $\alpha=2$. In bounded smooth domains $\Omega$, existence  and uniqueness  were established originally by 
		Krasnoselskii \cite{Kr}*{Sec. 7.2.6}, Brezis and Oswald \cite{BO}, and others, for more general concave nonlinearities, but under heavy 
		restrictions on classes 
		of solutions, coefficients and data.  (See \cite{BoOr},  \cite{CV1}, \cite{SV}, 
		 and the literature cited there.)  
		
		On the entire space $\mathbb{R}^n$, sharp existence and uniqueness 
		results, along with certain global pointwise estimates, were obtained by Brezis and Kamin \cite{BK} for \textit{bounded}  solutions $u>0$ to the equation $-\Delta u = \sigma u^q$. The proof of the uniqueness property given in \cite{BK} 
		for bounded solutions $u$ such  that $\liminf_{x\to \infty}u(x)=0$ is especially subtle; it is based on an analysis of a related  parabolic  porous medium equation.

	In this paper, we obtain bilateral pointwise estimates for \textit{all} solutions 
	to \eqref{frac_lap_eqn} using nonlinear potential theory. As a result, we solve the uniqueness problem and   give sharp 
		existence criteria, for arbitrary $\mu$ and $\sigma$, 
		in bounded uniform domains $\Omega$ for $0<\alpha\le 2$, and on the entire space $\mathbb{R}^n$ 
		for $0<\alpha<n$, as well as on complete, non-compact Riemannian manifolds $M$ with nonnegative Ricci curvature (see the Example below).

		If  $(-\Delta)^{\frac{\alpha}{2}}$  has a positive Green's function $G$ in  $\Omega$, then applying  Green's operator $\G$ to both sides, we obtain an equivalent problem where 
		solutions $u$ satisfy the integral equation 
		\eqref{sublin-eq-f} 
		with $f= \G \mu$. 
	 In the case $\alpha=2$, 
	 such solutions $u$ to \eqref{frac_lap_eqn} 
	 in  bounded $C^2$-domains $\Omega$ 
	 are usually called \textit{very weak} solutions  (see  \cite{FV1}, \cite{MV}).

		We obtain sharp \textit{lower estimates} of solutions  to \eqref{sublin-eq-f} for quasi-metric kernels $G$, and more 
		generally,  for quasi-symmetric (QS) 
	 kernels  $G$  which  satisfy  
		the weak maximum principle (WMP). Many examples 
	of elliptic differential operators whose  Green's kernels have these properties are given in \cite{An}.

	We also obtain matching \textit{upper  estimates} of solutions  for  quasi-metric, or quasi-metrically modifiable kernels $G$. In particular,  Green's kernels   for $(-\Delta)^{\frac{\alpha}{2}}$, $0<\alpha \le 2$, in \textit{uniform domains} $\Omega$   are quasi-metrically modifiable.  Hence, our general results  yield 
	 bilateral pointwise estimates of all solutions  to \eqref{frac_lap_eqn} in this case.

	When $2 < \alpha < n$, we can treat Green's kernels for nice domains $\Omega\subseteq \R^n$, such as the balls or half-spaces, where the Green kernel is known to be quasi-metrically modifiable  
	(see \cite{FNV}, \cite{FV2}, \cite{FV3}). 
	
	On the entire space $\Omega = \R^n$,  the Green kernel, i.e., the Newtonian kernel if $\alpha=2$, $n \ge 3$, and the Riesz kernel of order $\alpha$ if $0<\alpha <n$, are quasi-metric. Sublinear equations  \eqref{frac_lap_eqn} in this case were 
	treated earlier in \cite{CV1}, \cite{V2}. 		
	
For  $\sigma \in \mathcal{M}^+(\Omega)$ and  $0 < q < 1$, we consider weighted norm inequalities of $(1,q)$-type, 
\begin{equation}\label{main_ineq}
		\Vert \G \nu \Vert_{L^q(\Omega, \sigma)} \le C \, \Vert \nu \Vert, \qquad \forall \nu \in \mathcal{M}^+(\Omega), 
	\end{equation}
	 where we use the notation $\Vert\nu\Vert=\nu(\Omega)$ for $\nu\in \mathcal{M}^+(\Omega)$. We denote by  $\varkappa=\varkappa(\Omega, \sigma)$  the \textit{least constant} $C$ in \eqref{main_ineq}.

	 It is more convenient to use here $\mathcal{M}^+(\Omega)$ in place of $L^1(\Omega, \sigma)$. The latter corresponds to the $(1, q)$-type inequality
	\begin{align}
		\label{main_ineq-1}
		\Vert \G^\sigma f \Vert_{L^q(\Omega, \sigma)} \le C \, \Vert f \Vert_{L^1(\Omega, \sigma)}, \qquad \forall f  \in L^1(\Omega, \sigma). 
	\end{align} 
	
	\smallskip
	\noindent
	\textbf{Remark.} 	It follows from \cite{G}*{Lemma 3.I}  and  \cite{QV2}*{Theorem 1.1} that   \eqref{main_ineq}$\Longleftrightarrow$\eqref{main_ineq-1} 
	for  (QS)\&(WMP) kernels $G$. On the other hand, simple examples show that 
	\eqref{main_ineq-1}$\Longrightarrow$\eqref{main_ineq} may fail even for symmetric kernels without the (WMP) restriction. 
	\smallskip
	 
	 We observe that inequality \eqref{main_ineq-1} is the end-point case $p=1$ of the 
	$(p, q)$-type weighted norm inequality  
		\begin{equation}\label{p-r}
		\Vert \G^\sigma f \Vert_{L^q(\Omega, \sigma)} \le C \, \Vert f \Vert_{L^p(\Omega, \sigma)}, 
		\quad \forall f \in L^p(\Omega, \sigma), 
	\end{equation}
	when $0<q<p$ and $p \ge 1$.  
	
	For $p>1$, $0<q<p$, inequality \eqref{p-r}
		was characterized recently in \cite{V1}, in the context of   
	studying solutions $u \in L^{r}(\Omega, \sigma)$ with $r>q$ to the homogeneous 
	equation \eqref{sublin-eq-f} with $f=0$. The case $p=1$, or $r=q$,  is substantially more complicated.

	In \cite{QV2}*{Theorem 1.1}, we proved that \eqref{main_ineq} is equivalent  to the existence of a nontrivial supersolution $u\in L^q(\Omega, \sigma)$ to the homogeneous equation,  so that 
	\begin{equation}
			\label{super-sol} 
			 \G(u^q d \sigma)\le u<+\infty \quad d \sigma\textrm{-a.e.}  \, \,  \text{ in $\Omega$},  
				\end{equation}
	for kernels $G \ge 0$  that satisify (QS)\&(WMP) conditions. This can be viewed as a sublinear  version of Schur's lemma  (see \cite{G}).

	For a  kernel $G$, we 
	set 
	 \begin{equation}\label{qm_ball}
	 B(x, r)\defeq\{ y\in \Omega\colon \, G(x, y)>\tfrac{1}{r}\}, \qquad  x\in \Omega, \, \, r>0.
	 \end{equation}
	 
	 If $G$ is a quasi-metric kernel, then $B(x, r)$ is a \textit{quasi-metric ball}. 
	 
	 By Fubini's theorem we have 
	 \begin{equation}
			\label{lin_pot}
	 \G\sigma (x) = \int_0^\infty \frac{\sigma(B(x, r))}{r^2} \, d r, \qquad x \in \Omega. 
	 \end{equation}

	 Let  $E\subset \Omega$ be a Borel set. By $d \sigma_{E} = \chi_{E} \, d \sigma$ we denote the restriction  of $\sigma$ to $E$.  
	 We consider a localized version of inequality 
	 \eqref{main_ineq}, 
	 \begin{align}
		\label{main_ineq_loc}
		\Vert \G \nu \Vert_{L^q(\Omega, \, \sigma_{E})} \le  C \, \Vert \nu \Vert, \qquad \forall \nu \in \mathcal{M}^+(\Omega), 
	\end{align}
	
	By $\varkappa(E)=\varkappa(E, \sigma)$ we denote the least constant  $C$ in \eqref{main_ineq_loc}. 
In most cases, it suffices to consider the constants $\varkappa(B)=\varkappa(B, \sigma)$  for 
	 ``balls'' $B=B(x,r)$ defined by \eqref{qm_ball}. The dependence on $\sigma$ will often be dropped, particularly when $\sigma$  is the measure in 
	 \eqref{sublin-eq-f}.
	 Notice that, obviously,  
	 \begin{equation}\label{varkappa-def}
	 \varkappa (B\cap E, \sigma) =\varkappa (B, \sigma_E). 
 	 \end{equation}

	 We remark that, for (QS)\&(WMP) kernels, there are estimates of $\varkappa (B)$ 
 in terms of the norms of 
	 potentials $\G \sigma_{B}$ in  Lorentz spaces  	(\cite{QV2}, Theorem 1.2):  
	  \begin{align*}
		\label{lorentz}
	 C_1 \,   \Vert \G \sigma_{B} \Vert_{L^{\frac{q}{1-q}}(\Omega, \, \sigma_{B})}  \le  \varkappa(B)
		 \le   C_2 \, 	 \Vert \G \sigma_{B} \Vert_{L^{\frac{q}{1-q}, q}(\Omega, \, \sigma_{B})}, 
	\end{align*}
	 where  $C_1$ depends only on $q$ and  $\mathfrak{b}$, and $C_2$  on $q$, $\mathfrak{a}$ and  $\mathfrak{b}$. Here $\mathfrak{a}$ and  $\mathfrak{b}$ are
	 the constants in the 
	 conditions (QS) and (WMP),  respectively.  
	 
	 Using the constants $\varkappa (B)$, 
	  we construct the nonlinear potential $ \K \sigma$,  intrinsic to 
	 sublinear problems, by 
	 \begin{equation}
			\label{nonlin_pot}
	 \K\sigma (x) = \int_0^\infty \frac{\left[{\varkappa(B(x, r))}\right]^{\frac{q}{1-q}}}  {r^2}\, d r, \qquad x \in \Omega. 
	 \end{equation}
	 
	 We remark that nonlinear potentials of this type  for Riesz kernels $G(x,y)=|x-y|^{\alpha-n}$ on 
	 $\mathbb{R}^n$    
	 were introduced for the first time in \cite{CV1}. They are related to nonlinear potentials of 
	 Havin--Maz'ya--Wolff type, which appeared originally  
	 in the paper of Havin and Maz'ya \cite{MH}, and were used extensively by Hedberg and Wolff 
	 \cite{HW} 
	 (see also \cite{AH}, \cite{Maz}, and the literature cited there).

Let $\mu, \sigma \in \mathcal{M}^{+}(\Omega)$ ($\sigma\not=0$) and $0<q<1$.  
A Borel measurable function $u\colon \Omega \to [0, +\infty]$ is called a nontrivial \textit{supersolution}  associated with the equation 
\begin{equation}
			\label{sublin-sigma-mu} 
			u =  \G(u^q d \sigma) + \G \mu \quad  d\sigma\text{-a.e.} \, \, \text{in}  \, \, \Omega,  
			\end{equation}
if $u>0$ $d\sigma$-a.e., and 
\begin{equation}
			\label{super-sol-def} 
\G(u^q d \sigma) + \G \mu \le u <+\infty \quad d\sigma\text{-a.e.} \, \, \text{in}  \, \, \Omega.   
\end{equation}		
	 
	 A \textit{subsolution} is defined similarly as a Borel measurable function $u\colon \Omega \to [0, +\infty]$ such that 
	 	 \begin{equation}
			\label{sub-sol} 
u \le \G(u^q d \sigma) + \G \mu<+\infty \quad d\sigma\text{-a.e.} \, \, \text{in}  \, \, \Omega.   
\end{equation}		

A  nontrivial \textit{solution} to \eqref{sublin-sigma-mu} is both a subsolution and a nontrivial supersolution. If $u$ is a (super) solution, then $u\in L^q_{{\rm loc}} (\Omega, \sigma)$  (see Lemma \ref{local_lemma} below).

	We now state our main theorem for quasi-metric kernels.

	\begin{theorem}\label{strong-thm}
		Let $\mu, \sigma \in \mathcal{M}^{+}(\Omega)$ ($\sigma\not=0$) and $0<q<1$.  Suppose $G$ is a  quasi-metric kernel. Then the following statements hold. 
		
		{(i)} Any nontrivial solution $u$ to equation \eqref{sublin-sigma-mu} 
					satisfies the bilateral pointwise estimates  
				\begin{equation}
				\label{sublin-low} 
		 c \,   [ (\G \sigma(x))^{\frac{1}{1-q}} + 	 \K\sigma (x)] +  \G \mu(x) \le 	u (x)
		 \end{equation}	
		 and 
		 \begin{equation} \label{sublin-up} 
		u(x)  \le  \,  C \, [ (\G \sigma(x))^{\frac{1}{1-q}} + 	 \K\sigma (x) +  \G \mu(x)],   
				\end{equation}	
		$d \sigma$-a.e. in $\Omega$, 
		where $c, C$ are positive constants which depend only on $q$ and the 
		quasi-metric constant $\kappa$ of the kernel $G$. Moreover, such a solution $u$ is unique.

		{(ii)}  Estimate \eqref{sublin-low}  holds for any nontrivial supersolution $u$      
		at all $x \in \Omega$ such that  
		\begin{equation}
			\label{sublin-low-x} 
u(x) \ge \G(u^q d \sigma)(x) + \G \mu(x). 
\end{equation}		
		Similarly, 
		 \eqref{sublin-up}  holds for any subsolution $u$   at all  $x \in \Omega$ such that 
		\begin{equation}
			\label{sublin-up-x} 
u(x) \le \G(u^q d \sigma)(x) + \G \mu(x). 
\end{equation}

		{(iii)} A nontrivial (super) solution $u$ to \eqref{sublin-sigma-mu}  exists if and only if 
		the following three conditions hold:
		\begin{align}
	& \int_a^\infty \frac{\sigma(B(x_0, r))}{r^2} \, d r < \infty,  \label{cond-a} \\ 
	& \int_a^\infty \frac{\left[{\varkappa(B(x_0, r))}\right]^{\frac{q}{1-q}}}  {r^2}\, d r<\infty,
	\label{cond-b} \\ & \int_a^\infty \frac{\mu(B(x_0, r))}{r^2} \, d r < \infty,  
				 \label{cond-c}  
				 \end{align}	 
	for some (or, equivalently, all)  $x_0\in \Omega$ and $a>0$.	  Any nontrivial solution $u$ satisfies  \eqref{sublin-low}, \eqref{sublin-up} 
		 at all $x\in \Omega$ such that 
		\begin{equation}
			\label{sublin-eq-x} 
u(x) = \G(u^q d \sigma)(x) + \G \mu(x). 
\end{equation}	 	
	\end{theorem}
	
	\smallskip
	\noindent
	\textbf{Remarks.} 
	1. If $u$ is a nontrivial solution to \eqref{sublin-sigma-mu}, understood $d \sigma$-a.e., then  
	\[
	\tilde u(x)  \defeq \G(u^q d \sigma)(x) + \G \mu(x)
	\]
	 is a nontrivial solution to \eqref{sublin-eq-x} for all 
	$x \in \Omega$. Notice that $\tilde u = u$ $d \sigma$-a.e., and consequently 
	\[
	\tilde u(x)  = \G(\tilde u^q d \sigma)(x) + \G \mu(x), \quad \forall x \in \Omega. 
	\]
	Such representatives 
	$\tilde u$  can be used to obtain estimates of solutions that hold everywhere in $\Omega$. (See \cite{FV2} in the linear case $q=1$.) 
	\smallskip
	
	2. Under the assumptions of Theorem \ref{strong-thm}, conditions \eqref{cond-a}--\eqref{cond-c} hold if and only if  $\G \sigma <+\infty$, $\K \sigma <+\infty$, and $\G \mu < +\infty$ 
		 $d\sigma$-a.e., or equivalently $\G \sigma\not\equiv+\infty$, $\K \sigma \not\equiv+\infty$, and $\G \mu \not\equiv+\infty$. Another existence 
		 criterion is given below (see Lemma \ref{equiv-lemma} and Corollary \ref{equiv-cor}). 	\smallskip
	
	3. An analogue of Theorem \ref{strong-thm} holds for  equation \eqref{sublin-eq-f} 
	with arbitrary Borel measurable function $f\ge 0$ (see Theorem \ref{main-thm-f} below). One only needs to replace $\G \mu$ with $\G(f^q d \sigma)+ f$ in \eqref{sublin-low}, 
	\eqref{sublin-up}, and the corresponding  estimates for  sub/super solutions. Notice that in the special case $f=\G \mu$ the extra term $\G(f^q d\sigma)$ may  be dropped. 
	\smallskip
	

	\smallskip
	\noindent
	\textbf{Example.} Let $M$ be a complete, non-compact Riemannian manifold with the volume doubling condition. If the minimal Green's function $G$  
	 satisfies the Li--Yau estimates, then $G$ is known to be a quasi-metric kernel  \cite{GSV}*{Lemma 6.1}. In particular, this is true on manifolds $M$ with nonnegative Ricci curvature, and in many other 
	  circumstances (see \cite{GSV}). 
 Under these assumptions, Theorem \ref{strong-thm} gives existence, uniqueness, and bilateral estimates of positive solutions to the sublinear 
	 elliptic equation  $-\Delta u = \sigma u^q + \mu $, where $\Delta$ is the Laplace--Beltrami operator on $M$. 
	 	\smallskip

			As was mentioned above, in many applications	the kernel $G$ can  be modified 
	using a positive function $m\in C(\Omega)$, which is called a \textit{modifier}.  If the modified kernel 
		\begin{equation}
		\label{mod-ker}
		\widetilde{G}(x, y) = \frac{G(x, y)}{m(x) \, m(y)}, \quad x, y \in \Omega,
		\end{equation}
	 satisfies the (QS)\&(WMP), then our lower estimates of solutions are applicable (see Sec. \ref{q-m-m kernels}). Upper estimates require stronger assumptions.

	\smallskip
	\noindent
	\textbf{Definition.} A kernel $G$ on $\Omega \times\Omega$ is said to be quasi-metrically modifiable, with modifier $m$, if the kernel $\widetilde{G}$ defined by \eqref{mod-ker}
	is a quasi-metric kernel with quasi-metric constant $\tilde{\kappa}$. 
	\smallskip

	For quasi-metrically modifiable kernels $G$,  Theorem~\ref{strong-thm}  is applicable 
		with $\widetilde{G}$ in place of $G$, which leads to 
		matching  lower and upper global  estimates of solutions up to the boundary of $\Omega$. 
		
		A typical modifier    is given by 
	\begin{equation}
		\label{typ-mod}
g(x) = \min \{1, G(x, x_0)\}, \quad x \in \Omega,
	\end{equation}
	where $x_0$ is a fixed pole in $\Omega$, provided $g\in C(\Omega)$. 
	(Sec. \ref{q-m-m kernels} below; see also \cite{HN}, Sec. 8.)   
	
	In particular, this procedure 
	is applicable to  Green's kernels $G$  for $(-\Delta)^{\frac{\alpha}{2}}$ in certain  domains $\Omega\subset \R^n$,  for instance, balls or half-spaces, if  $0<\alpha <n$, or uniform domains discussed below 
	if $0<\alpha\le 2$.   For bounded $C^{1,1}$-domains $\Omega$ if $0<\alpha\le 2$, 
	as well as  balls or half-spaces if  $0<\alpha <n$, it is well known that $g(x) \approx [\textrm{dist}(x, \Omega^c)]^{\frac{\alpha}{2}}$.

	This approach works also for Green's functions of 
	uniformly elliptic, symmetric operators $L$ in divergence form, 
	\begin{equation}
		\label{typ-L}
		L u=\textrm{div} (A \nabla u), \quad A=(a_{ij}(x))_{i, j=1}^n, \quad a_{ij}=a_{ji} \in L^\infty(\Omega),
			\end{equation}
with real-valued coefficients $a_{ij}$ in place of the Laplacian 
	 (see \cite{FV2}, \cite{H}, and the literature cited there).

	 Suppose $G$  is a quasi-metrically modifiable kernel, with modifier $m$, associated with the quasi-metric 
	$\widetilde{d}=1/\widetilde{G}$. 		
	We denote by $\widetilde{B}(x, r)$  a quasi-metric ball 	
	\begin{equation}\label{def-tilde-mod} 
	\widetilde{B}(x, r) \defeq \left\{ y \in \Omega\colon \, \, \widetilde{G}(x,y)>\tfrac{1}{r} \right\}, \quad x \in \Omega, \, \, r>0. 
		\end{equation}

	Let $d \tilde \sigma=m^{1+q} d \sigma$.  For a Borel set $E \subseteq \Omega$, 
	 by $\tilde{\varkappa}(E)=\tilde{\varkappa}(E, \tilde \sigma)$ we denote the least constant in the inequality 
	\begin{equation}\label{def-kap-K} 
	 		\Vert \widetilde{\G} \nu \Vert_{L^q(\Omega, \widetilde{\sigma}_E)} \le \tilde{\varkappa}(E) \, \Vert \nu\Vert,
			\qquad \forall  
			\, \,  \nu \in \M(\Omega).
				 	\end{equation}

	Using the constants $\tilde{\varkappa}(\widetilde{B}(x, r))$, we construct the modified intrinsic potential $\widetilde{\K} \sigma$  defined by 
		\begin{equation}\label{def-K-mod} 
	\widetilde{\K} \sigma  (x) \defeq \int_0^\infty \frac{ [\tilde{\varkappa}(\widetilde{B}(x, r))]^{\frac{q}{1-q}} }{r^2} \, dr, \qquad x\in \Omega.
		\end{equation}

	\begin{theorem}\label{strong-thm-mod}
		Let $\mu, \sigma \in \mathcal{M}^{+}(\Omega)$ ($\sigma\not=0$) and $0<q<1$.  Suppose 
		$G$ is a  quasi-metrically modifiable kernel with modifier  $m$.		Then any nontrivial solution $u$ to equation \eqref{sublin-sigma-mu} is unique and 
					satisfies the bilateral pointwise estimates  
				\begin{equation}
				\label{sublin-low-mod} 
		 c \,   m \,  \left( \left[ \frac{\G (m^q d\sigma)}{m}\right]^{\frac{1}{1-q}} + 
		   \widetilde{\K} \sigma\right) +   \, \G \mu\le 	u 
		   \end{equation}
		   and
		\begin{equation}\label{sublin-up-mod} 
		u \le  \,  C \, m \,  \left( \left[ \frac{\G (m^q d\sigma)}{m} \right]^{\frac{1}{1-q}}  + 	  \widetilde{\K} \sigma\right) +   C \,  \G \mu,   
				   \end{equation}	
		$d \sigma$-a.e. in $\Omega$, 
		where $c, C$ are positive constants which depend only on $q$ and the 
		quasi-metric constant $\tilde{\kappa}$ of the modified kernel $\widetilde{G}$. 	
		
		The lower bound 	\eqref{sublin-low-mod} holds for any nontrivial supersolution $u$, whereas the upper bound \eqref{sublin-up-mod}  holds for any subsolution $u$. 
		\end{theorem}

\smallskip
	\noindent
	\textbf{Remarks.} 1. Under the assumptions of Theorem \ref{strong-thm-mod}, a 
	 nontrivial (super) solution to \eqref{sublin-sigma-mu} exists  if and only if 
	 $\G (m^qd\sigma) <+\infty$, $\widetilde{\K} \sigma <+\infty$, and $\G \mu < +\infty$ 
		 $d\sigma$-a.e., or equivalently (see Sec. \ref{q-m-m kernels}) 
			\begin{equation} \label{cond-a-mod}
		  \G (m^qd\sigma)  \not\equiv+\infty, \quad \widetilde{\K} \sigma \not\equiv+\infty, \quad 
		  \G \mu \not\equiv+\infty. 
	\end{equation}

	2. If $G$ is quasi-metrically modifiable with  modifier  $m=g$ given by \eqref{typ-mod}, 
	a nontrivial  (super) solution to \eqref{sublin-sigma-mu} exists if and only if  
	\begin{equation} \label{tilde-kappa}
		\tilde{\varkappa} (\Omega)<\infty \quad \textrm{and}    \quad 	\int_\Omega g  \, d \mu<\infty.  \end{equation}

3.	The lower bound \eqref{sublin-low-mod}	 holds for any nontrivial supersolution $u$ if $\widetilde{G}$ is a (QS) kernel which satisfies the  (WMP); see Sec. \ref{q-m-m kernels}  
	below. 
	
\smallskip

	 In the following definition of a \textit{uniform domain} (or, equivalently, an interior NTA domain), we rely on the notions of the interior corkscrew condition and the Harnack chain condition. We refer to 
   \cite{H} for related definitions, in metric spaces, along with 
	a discussion of quasi-metric properties, 3-G inequalities,   and the uniform boundary Harnack principle (see also \cite{Ai}).

	\smallskip
	\noindent
	\textbf{Definition.} A uniform domain $\Omega\subset \R^n$, $n \ge 2$,  
is a bounded domain which 
satisfies  the interior corkscrew condition and the Harnack chain condition. 
	\smallskip

	Notice that uniform domains are not necessarily regular in the sense of Wiener. Bounded Lipschitz and non-tangentially accessible (NTA) domains are examples of regular uniform domains.
	
	The next corollary  is a direct consequence of Theorem \ref{strong-thm-mod} and the fact that Green's 
 function $G$ of  $(-\Delta)^{\frac{\alpha}{2}}$  in  a uniform domain   
 $\Omega$ for  $0<\alpha \le 2$  is quasi-metrically modifiable, with $m=g$ and quasi-metric constant $\tilde{\kappa}$ 
 which does not depend on the choice of $x_0\in \Omega$  
  (see \cite{An}, \cite{H}).

	\begin{cor}\label{appl-unif} Suppose $\Omega\subset \R^n$, $n \ge 2$,  is a uniform domain. Suppose   $G$ is   Green's kernel of $(-\Delta)^{\frac{\alpha}{2}}$  in $\Omega$, where $0<\alpha \le 2$,   
	 $\alpha <n$. Define the modifier $m=g$  by 
 \eqref{typ-mod} with pole $x_0\in \Omega$. 
	
	  Let $0<q<1$, and let $\mu, \sigma \in \M(\Omega)$, and $d \tilde \sigma = 
	g^{1+q} d \sigma$.   
Then the following statements hold. 
 
  (i) Any nontrivial solution $u$ to equation \eqref{frac_lap_eqn} is unique and 
 satisfies estimates \eqref{sublin-low-mod}, \eqref{sublin-up-mod}  $d \sigma$-a.e., 
 and at all $x \in \Omega$ where \eqref{sublin-eq-x} holds. 
 
 (ii) Any nontrivial supersolution $u$ satisfies the lower bound \eqref{sublin-low-mod},  and any  subsolution $u$ satisfies the upper bound \eqref{sublin-up-mod}.

(iii) A nontrivial (super) solution to  \eqref{frac_lap_eqn}  exists if and only if \eqref{tilde-kappa}
holds,  for some (or, equivalently, all) $x_0\in \Omega$.	 
\end{cor}

\smallskip
	\noindent
	\textbf{Remarks.} 1. In the case $0<\alpha<2$, $n \ge 2$, Corollary \ref{appl-unif} holds in any  bounded domain $\Omega\subset \R^n$ with the interior corkscrew condition, without requiring that the Harnack chain condition holds.  When $n=\alpha=2$, Corollary \ref{appl-unif} holds 
	in any finitely connected domain $\Omega\subset \R^2$ with positive Green's function, in particular, in a bounded Lipschitz domain. In all of these 
	cases,  Green's function $G$ of $(-\Delta)^{\frac{\alpha}{2}}$ is known to be quasi-metrically modifiable with modifier $g$ (see 
	\cite{H}*{Sections 3 and 4}). \smallskip
	
	2.  Corollary \ref{appl-unif} holds  for uniformly  elliptic operators $L$ in divergence form  
 given  by \eqref{typ-L},  
 in place of the Laplacian $\Delta$, in NTA domains, as well as uniform domains with Ahlfors regular boundary  (see \cite{FV2}, \cite{H}  and the references given there). \smallskip
	
	3. Uniqueness of solutions to  sublinear problems of the type  \eqref{frac_lap_eqn} was 
	previously known only under various restrictions on solutions, coefficients, and data, for instance, for bounded solutions   \cite{BK}, \cite{BO}, or finite energy solutions \cite{SV}.\smallskip

	I am thankful to Fedor Nazarov for valuable comments, and particularly 
	for observing that the uniqueness property of solutions to sublinear problems  is an immediate consequence of the bilateral pointwise estimates.  
	
	\section{Kernels and potential theory}\label{background}

Let $\Omega$ be a locally compact Hausdorff space with countable base.
	In this section, we consider nonnegative lower semicontinuous kernels $G\colon \Omega \times \Omega \rightarrow (0, +\infty]$ (see \cite{Brelot}, \cite{Fug}, \cite{Fug65}).
		 
	 By $S_\mu$ we denote the closed support of $\mu \in \mathcal{M}^+(\Omega)$. We  set $\Vert \mu \Vert \defeq \mu (\Omega)$. For $\mu \in \mathcal{M}^+(\Omega)$,  the potential $\G\mu$, and the 
	 adjoint potential $\G^*\mu$, are defined, respectively, by 
		\begin{align*}
		& \G\mu (x)  \defeq \int_{\Omega} G(x,y) \, d\mu(y), \quad \forall x \in \Omega,  \\
& \G^*\mu (x)  \defeq \int_{\Omega} G(y, x) \, d\mu(y), \quad \forall x \in \Omega. 
\end{align*}

	\smallskip
	\noindent
	\textbf{Definition.} The kernel $G$ on $\Omega\times \Omega$ 
	satisfies the \textit{Weak Maximum Principle} (WMP), with constant $\mathfrak{b}\ge 1$, if 
			\begin{equation}\label{wmp-def} 
			\G\mu (x) \le 1, \quad \forall x \in S_\mu \Longrightarrow \G\mu (x) \le \mathfrak{b},   \quad \forall x \in \Omega, 
			\end{equation}
		for any $\mu \in \mathcal{M}^+(\Omega)$. 
		
		When $\mathfrak{b}=1$, we say that $G$ satisfies the (Frostman) \textit{Maximum Principle}.
			\smallskip
	
The following lemma    was stated in  \cite{QV2}*{Lemma 3.5}. However, the proof given there is valid  
only under the usual assumptions of potential theory, namely, 
that the kernel $G$ is finite off the diagonal and continuous in the extended sense in  $\Omega\times \Omega$. A complete 
proof for general quasi-metric kernels is given below.   
				\begin{lemma}
		\label{qmm-wmp}
		Let $G$ be a quasi-metric kernel with quasi-metric constant $\kappa$.
		Then $G$ satisfies the (WMP) with constant $\mathfrak{b}  = 2\kappa$.  
	\end{lemma}	
	
	\begin{proof} Without loss of generality we may assume that the measure  $\mu\in \M(\Omega)$ in 
	\eqref{wmp-def}  is compactly supported. Suppose that $\G \mu\le 1$ on $K\defeq S_\mu$.
	Let us  fix $x \in \Omega\setminus K$, and set 
	\[
	\rho(x) \defeq \inf\{d(x, z)\colon \, z\in K\}, 
	\]
	where $d=1/G$. 
	
	We first consider the case $\rho(x) >0$. As in the proof of \cite{QV2}*{Lemma 3.5}, for any $\epsilon>0$ we choose $y'\in K$ for which 
	$d(x, y') < (1+ \epsilon) \rho(x)$.  It follows 
	that  $d(x, y') < (1+ \epsilon) \, d(x, y)$ for any $y\in K$, since   $\rho(x)\le d(x, y)$. Then  
	\begin{equation*}
	d(y', y) \le  \kappa \,  [ d(x,y)+ d(x, y')]  < \kappa (2 + \epsilon) \, d(x,y). 
		\end{equation*}
		Hence, $G(x,y) < \kappa (2 + \epsilon) \, G(y', y)$ for all $y \in K$, which yields 
		\[
		\G \mu(x) \le \kappa (2 + \epsilon) \,  \G \mu(y') \le \kappa (2 + \epsilon). 
		\]
		Letting $\epsilon \to 0$ proves \eqref{wmp-def} with $\mathfrak{b}  = 2\kappa$ if $\rho(x)>0$.  
		
		In the case $\rho(x)=0$, we set $E_0(x)\defeq \{z\in K \colon \, d(x, z) =0\}$. 
		Clearly,  $E_0(x)$ is a Borel subset  of $K$.   
		
		If $E_0(x)\not=\emptyset$, then there exists $y' \in K$ such that $d(x, y')=0$, so that 
	\[
		d(y', y)  \le \kappa \,  [ d(x, y)+ d(x, y')]  \le \kappa \,  d(x,y),  
		\]
		for all $y\in K$. Hence, $G(x, y)\le \kappa \, G(y', y)$,  and consequently 
		\[
		\G \mu(x) \le \kappa \, \G \mu(y') \le \kappa. 
		\] 
			
		In the case $E_0(x)=\emptyset$, there exists a sequence of points $y_m\in K$ such that 
		$d (x, y_m)>0$ and $d(x, y_m) \downarrow 0$ as $m\to \infty$. Set  
		 $a_m \defeq  1/d (x, y_m)$. For all $y\in K$, we have 
		 \[
		 d(y_m, y)\le \kappa \, [ d(x, y)+ d(x, y_m)] \le 2 \kappa \max \{ d(x, y), d(x, y_m)\}.
		 \]
		Hence, $\min \{ G(x, y), a_m\} \le 2 \kappa \, G (y_m, y)$, and consequently 
		\[
		\int_{\{ y\in K \colon G(x, y)\le a_m\}} G(x, y) \, d\mu(y) \le 2 \kappa \,  \G\mu (y_m) \le  2 \kappa.
		\]
		Since   $a_m\uparrow \infty$ and $E_0(x)=\emptyset$,  the monotone convergence theorem  gives  
		\[
		\G \mu (x) \le 2 \kappa. 
		\]
		This 
completes the proof in the case $\rho(x)=0$. 
	\end{proof}

Let  $0 < q < 1$ and  $\sigma \in \mathcal{M}^+(\Omega)$. Suppose $G$ is a kernel in $\Omega\times \Omega$. We consider nontrivial solutions $u>0$ $d \sigma$-a.e. to the 
\textit{homogeneous} integral equation
		\begin{align}\label{int-eq}
			u = \G(u^q d\sigma)< \infty \quad  d\sigma\textrm{-a.e.} \, \, \text{in} \, \, \Omega,
		\end{align}
		We also study nontrivial \textit{supersolutions} $u>0$ $d \sigma$-a.e. to the corresponding integral inequality 
		\begin{align}\label{int-sup}
			\G(u^q d \sigma)\le u  <\infty \quad d\sigma\textrm{-a.e.} \, \, \text{in} \, \, \Omega,
		\end{align} 
		and  \textit{subsolutions} $u$ such that  
		\begin{align}\label{int-sub}
			0\le  u \le \G(u^q d\sigma) <\infty \quad d\sigma\textrm{-a.e.} \, \, \text{in} \, \, \Omega.
		\end{align}

	The following lemma  proved in \cite{QV2}*{Lemma 2.2} shows that,  if there exists a  (super) solution $u<+\infty$ 
		$d\sigma$-a.e., then  actually $u \in L^q_{\rm loc}(\Omega, \sigma)$.

	\begin{lemma}
		\label{local_lemma}
		Let $G$ be a  kernel  on $\Omega \times \Omega$.
		Suppose $u\ge0$ is a  supersolution such that \eqref{int-sup} holds.  
		Then $u \in L^q_{\rm loc}(\Omega, \sigma)$.  
	\end{lemma}
	
	\smallskip
	\noindent
	\textbf{Definition.} A kernel $G$ on $\Omega \times \Omega$ is \textit{quasi-symmetric} (QS) 
	if,  for a positive constant $\mathfrak{a}$, 
		\[ \mathfrak{a}^{-1} G(y,x) \le G(x,y) \le \mathfrak{a} \, G(y,x), \quad \forall x, y \in \Omega. \]
	\smallskip
	
	A \textit{symmetrized} kernel $G^s$ is defined by
		\[ G^s(x,y) \defeq G(x,y) + G(y,x). \]
		Clearly,  $G^s$ is symmetric.  For a  (QS) kernel  $G$,   $G^s$ is comparable to $G$: 
 \[ \left(1 + \frac{1}{\mathfrak{a}}\right) \, G(x,y) \le G^s(x,y) \le (1 + \mathfrak{a}) \, G(x,y) , \quad x, y \in \Omega. \]
		
		We denote the integral operator with kernel $G^s$ by $ \G^s$. 
	For a  (QS) kernel  $G$, the least constants $\varkappa$ in the inequality
			\begin{align*}
				\Vert \G \nu \Vert_{L^q(\Omega, \sigma)} \le \varkappa \, \Vert \nu \Vert, \quad \forall \nu \in \mathcal{M}^+(\Omega), 
			\end{align*}
		and $\varkappa_s$ in the inequality
			\begin{align*}
				\Vert \G^s \nu \Vert_{L^q(\Omega, \sigma)} \le \varkappa_s \, \Vert \nu \Vert, \quad \forall \nu \in \mathcal{M}^+(\Omega), 
			\end{align*}	
		are equivalent: 
		\[ \left(1 + \frac{1}{\mathfrak{a}}\right) \, \varkappa \le \varkappa_s \le (1 + \mathfrak{a}) \varkappa. \]

If  $G$ is a (QS) kernel, then there is a nontrivial  supersolution $u>0$ $d  \sigma$ a.e. such that 
			$\G(u^q d \sigma) \le u<\infty$ $d  \sigma$ a.e. 
			if and only if there is a nontrivial supersolution $u_s\ge 0$ $d  \sigma$ a.e. to the symmetrized version,   
			$\G^s(u^q_s d \sigma)\le u_s<\infty$ $d  \sigma$ a.e.  
			This is easy to see using a scaled version $u_s= c_s \, u$ with an appropriate 
			positive constant $c_s$ which depends only on $\mathfrak{a}$ and $q$.

	  There are numerous notions of \textit{capacity} in analysis. For a discussion of 
	  $L^p$-capacities  ($1<p<\infty$) and the corresponding nonlinear potential theory, as well as capacities associated with  Sobolev  spaces and other function spaces we refer to \cite{AH}, \cite{Maz}, 
	  and the literature cited there. Capacity associated with 
 Green's function of the Laplacian in a domain $\Omega \subseteq \R^n$ (whenever  $\Omega$ admits a 
nontrivial Green's function)  is  
 fundamental to classical potential theory (see, e.g.,  \cite{Doob}).

We use capacities studied by Choquet in the framework of linear potential theory ($p=1$) for potentials $\G \mu$, 
 with kernel 
$G \colon \Omega\times \Omega \to [0, +\infty]$ and $\mu \in \M(\Omega)$ (see \cite{Brelot}, \cite{Fug}, \cite{Fug65}).

  For a compact set $K \subset \Omega$, 
	 the capacity $\capa_0(K)$ 
	 can be defined  as follows (see \cite{Brelot}, \cite{Fug65}), 
  	\begin{equation*}\label{cont-def}
		\capa_0 (K) \defeq \sup \{ \mu(K)\colon \,  \mu \in \M(K), \quad \G^*\mu(y) \le 1, \, \,  \forall \, y \in \Omega \}. 
	\end{equation*}

	We will mostly use the following version of capacity.
	
	\smallskip
	\noindent
	\textbf{Definition.} 	
	The \textit{Wiener capacity} $\capa (K)$ of a compact set $K \subset \Omega$ 
	is defined by 
		\begin{align} \label{capa}
			\capa (K) &\defeq \sup \, \{ \mu(K)\colon \mu \in \mathcal{M}^+(K), \quad \G^*\mu(y) \le 1,  \, \, \forall \, y \in S_\mu \}. 	
		\end{align}
	\smallskip

Obviously, $\capa_0(K) \le \capa(K)$, and for (WMP) kernels,    
\[
\capa_0(K) \le \capa(K) \le \mathfrak{b} \,  \capa_0(K).
\]

It is known that, for a kernel $G>0$, we have 
 $\capa(K) < + \infty$ for every compact $K \subset \Omega$ \cite{Fug}*{Sec. 2.5}.

The capacity $\capa$ 
 can be extended as an 
	 ``exterior'' set function, first to  open sets $B\subset \Omega$, and then to 
	arbitrary sets $A\subset \Omega$. In particular,  
	 \begin{align*}
	 & \capa(B) \defeq \sup \, \{ \capa(K)\colon \, \, \textrm{for all compact sets} \, K,  \, \, K\subset B\}, \\
	 & \capa(A) \defeq \inf \, \{ \capa(B)\colon \, \, \textrm{for all open sets} \, B, \, \, A\subset B \}.
	 \end{align*}

	A measure $\mu \in \M(\Omega)$ is \textit{absolutely continuous with respect to capacity} if 
	\[
		\capa(K) = 0 \Longrightarrow \mu(K) = 0, \quad  \textrm{for every compact set} \, \, K.
		\]

	The notions of  capacity and \textit{equilibrium measure}, i.e., an extremal measure in 
	\eqref{capa},  were an essential feature of
	  our approach in \cite{QV2},  which is developed further in this paper.    
	A proof of the following lemma can be found in \cite{QV2}*{Lemma 4.2}.

	\begin{lemma}\label{soln_abs_cont}
		Let $0<q< 1$ and $\sigma \in \M(\Omega)$. Let $G$ be a kernel on $\Omega\times \Omega$. Suppose $\G^*(u^q d \sigma) \le u$ $d\sigma$-a.e., where 
		$u \ge 0$, $u \in L^q_{{\rm loc}}(\Omega, \sigma)$. 
		Then $d\omega \defeq u^q d\sigma$ is absolutely continuous with respect to capacity.
		 
				If in addition $u > 0$ $d\sigma$-a.e., then $\sigma$ is absolutely continuous with respect to capacity.
	\end{lemma}

	\section{Lower bounds for supersolutions}\label{lower bounds}
	We will need the following lower bound for supersolutions obtained in \cite{GV}*{Theorem 1.3}. 
	\begin{lemma}
		\label{G-lemma-lower} Let $\sigma \in \M(\Omega)$ and $0<q<1$.  
		Suppose $G$ is a kernel on $\Omega \times \Omega$ which satisfies the (WMP) with constant 
		$\mathfrak{b}$. Then any nontrivial supersolution $u>0$ $d \sigma$-a.e. such that $\G(u^q d \sigma)\le u<+\infty$ $d \sigma$-a.e. satisfies the estimate
		\begin{equation}\label{G-lower-est} 
	 		u(x)\ge c \,  \left[\G\sigma(x)\right]^{\frac{1}{1-q}}, 
	 	\end{equation}
	where $c=(1-q)^{\frac{1}{1-q}} \mathfrak{b}^{-\frac{q}{1-q}}$, for all $x\in \Omega$ such that  $\G(u^q d \sigma)(x)\le u(x)$.
	\end{lemma}
	
Another lower estimate for supersolutions $u$, deduced in the next lemma, complements \eqref{G-lower-est} in a crucial way.   It holds 
for kernels $G$ which satisfy both the (WMP) and (QS) conditions. Using a symmetrized kernel, 
we may assume without loss of generality that $G$ is symmetric;  for (QS) kernels, the constant 
$C$  in \eqref{K-lower-est}   will depend on $q$, $\mathfrak{b}$, and the quasi-symmetry constant $\mathfrak{a}$. 
	\begin{lemma}
		\label{K-lemma-lower} Let $\sigma \in \M(\Omega)$ and $0<q<1$.  
		Suppose $G$ is a symmetric kernel on $\Omega \times \Omega$ which satisfies the (WMP) with constant 
		$\mathfrak{b}$. Then any nontrivial supersolution $u>0$ $d \sigma$-a.e.  such that $\G(u^q d \sigma)\le u<+\infty$ $d \sigma$-a.e. satisfies the estimate
		\begin{equation}\label{K-lower-est} 
	 		u(x)\ge c \,  \K \sigma(x), 
	 	\end{equation}
	where $c=(1-q)^{\frac{1}{1-q}} \mathfrak{b}^{-\frac{q}{1-q}}$, for all $x\in \Omega$ such that $\G(u^q d \sigma)(x)\le u(x)$.
	\end{lemma}
	
	\begin{proof} The proof 
of \eqref{K-lower-est}   makes use of an idea employed in the proof of \cite{QV2}*{Lemma 5.11}. 
		Let $u$ be a nontrivial supersolution. Then $u \in L^q_{\rm loc} (\Omega, \sigma)$ 
		by Lemma \ref{local_lemma}. 
		We set  
		 $d \omega \defeq u^q \, d\sigma$, so that $\omega \in \M(\Omega)$.
		
	For $x \in \Omega$ and $t>0$, let $B(x, t)$ be a ``ball'' defined by \eqref{qm_ball}. We first prove the estimate 
	\begin{equation}\label{kappa-lower-est} 
	\varkappa (B(x, t)) \le \frac{\mathfrak{b}}{(1-q)^{\frac{1}{q}}}  \, 
	\Vert u \Vert^{1-q}_{L^q (\Omega, \sigma_{B(x, t)})}, 
	\quad \forall \, x \in \Omega, \, \, t>0, 
		 	\end{equation}
	where without loss of generality we assume that   $\sigma(B(x,t))>0$  and  $\Vert u \Vert_{L^q (\Omega, \sigma_{B(x, t)})}<\infty$. 
	 We set  $d \mu= d \omega_{B(x, t)} = u^q d \sigma_{B(x, t)}$. 
	 
	 Suppose $\nu \in \M(\Omega)$ is a probability measure. Since $u  \ge \G \omega \ge \G \mu$, 
	 it follows that $\G \mu<\infty$ $d \mu$-a.e., and 
		\begin{align*}
			 \int_\Omega (\G\nu)^q \, d\sigma_{B(x, t)} &= \int_\Omega 
			\left( \frac{\G\nu}{u} \right)^q u^q \, d\sigma_{B(x, t)}  \\
				&\le \int_\Omega \left( \frac{\G \nu}{\G\mu} \right)^q \, d\mu.
		\end{align*}	
		 	For $\lambda>0$, we set $E_\lambda \defeq \left \{y \in B(x,t)\colon \,  \, \frac{\G \nu(y)}{\G\mu(y)}>\lambda \right\}$. 
			Then, for any $\beta>0$, we clearly have 	
		\begin{align*}		
			& \int_\Omega \left( \frac{\G \nu}{\G\mu} \right)^q \, d\mu = q \int_0^\infty \mu (E_\lambda) \,  \lambda^{q-1} \, d\lambda	  \\
				&= q \int_0^\beta \mu (E_\lambda) \,  \lambda^{q-1} \, d\lambda   + q \int_\beta^\infty \mu (E_\lambda) \, \lambda^{q-1} \, d\lambda \\& 
				\defeq I + II.
		\end{align*}
		Clearly,    
		 \[
		 I \le  q \, \Vert \mu \Vert  \int_0^\beta   \lambda^{q-1} \, d\lambda  = \beta^q \, \Vert \mu \Vert.
		 \]

		To estimate $II$, we use the (1,1) weak-type bound \cite{QV2}*{Lemma 5.10}, 
			\[ \mu(E_\lambda) \le \frac{\mathfrak{b} \,  \Vert \nu\Vert}{\lambda} = \frac{\mathfrak{b}}{\lambda}. 
			\] 
			Notice that by Lemma \ref{local_lemma} and Lemma \ref{soln_abs_cont},  $\omega$ is absolutely continuous with respect to capacity. 	Hence, the same is true for $\mu$. 

		It follows,   
		\[
		II \le q \, \mathfrak{b}  \int_\beta^{\infty} \lambda^{q-2} d \lambda =
		\frac{q}{1-q} \mathfrak{b} \beta^{q-1}.
		\]
		Choosing $\beta = \frac{\mathfrak{b}}{\Vert \mu\Vert}$, we deduce 
			\[ \int_\Omega (\G\nu)^q \, \, d\sigma_{B(x,t)}  \le \frac{\mathfrak{b}^q }{1-q} \, 
			\Vert \mu\Vert^{1-q}=
			\frac{\mathfrak{b}^q }{1-q} \left( \int_{B(x,t)}  u^q \, d\sigma  \right)^{1-q}. \]
		For a general (finite, nonzero)  measure $\nu \in \mathcal{M}^{+}(\Omega)$, 
		by homogeneity we obtain the  inequality
			\[ \int_{\Omega}  (\G\nu)^q \, d\sigma_{B(x, t)}  \le \frac{\mathfrak{b}^q }{1-q} \left( \int_{B(x,t)} u^q \, d\sigma \right)^{1-q} \Vert\nu\Vert^q, \]
			which proves \eqref{kappa-lower-est}. Therefore, 
		\begin{align*}
		\K \sigma (x)  & \defeq \int_0^\infty \frac{\left[\varkappa(B(x,t))\right]^{\frac{q}{1-q}}}{t^2} dt 
		 \le    \frac{\mathfrak{b}^{\frac{q}{1-q}}}{(1-q)^{\frac{1}{1-q}}} \int_0^\infty \frac{\int_{B(x,t)} u^q d \sigma}{t^2} dt \\
		& \,\,   = \frac{\mathfrak{b}^{\frac{q}{1-q}}}{(1-q)^{\frac{1}{1-q}}} \G(u^q d \sigma)(x) \le \frac{\mathfrak{b}^{\frac{q}{1-q}}}{(1-q)^{\frac{1}{1-q}}} u(x). 
		\end{align*}
			\end{proof}

	\begin{cor} 
	\label{cor-lower-est} Let $\mu, \sigma \in \M(\Omega)$ and $0<q<1$.  
		Suppose $G$ is a quasi-symmetric kernel on $\Omega \times \Omega$ 
		with quasi-symmetry constant $\mathfrak{a}$, which satisfies the (WMP) with constant 
		$\mathfrak{b}$. Then any nontrivial supersolution $u$ to \eqref{sublin-sigma-mu} satisfies the estimate
		\begin{equation}\label{K-lower-est-mu} 
	 		u(x)\ge c \,  [(\G \sigma(x))^{\frac{1}{1-q}}+\K \sigma(x)] + \G \mu(x), 
	 	\end{equation}
	where $c=c(q, \mathfrak{a}, \mathfrak{b})$, for all $x\in \Omega$ such that 
	\begin{equation}\label{x-lower-est-mu} 
	 	\G(u^q d \sigma)(x)+ \G \mu(x)\le u(x). 	
	 	\end{equation}
	In particular, \eqref{K-lower-est-mu} holds $d \sigma$-a.e.
	\end{cor} 
	
	\begin{proof} Suppose $u$ is a nontrivial supersolution  to \eqref{sublin-sigma-mu}. Let us set  
	$v(x)\defeq \G(u^q d \sigma)(x)$, $x \in \Omega$. Then, obviously,  $0<v \le u<\infty$ $d \sigma$-a.e. Hence, $v$ is a nontrivial 
	supersolution such that 
	\[
	\G(v^q d \sigma)(x) \le \G(u^q d \sigma)(x) =
	  v(x), \quad \forall \, x\in \Omega. 
	\]

		Then by  Lemma \ref{G-lemma-lower} and  Lemma \ref{K-lemma-lower}, we have 
	\begin{equation}\label{v-lower-est-mu} 
	 		v(x)\ge c \,  [(\G \sigma(x))^{\frac{1}{1-q}}+\K \sigma(x)], \quad \forall \, x\in \Omega, 
	 	\end{equation}
	where $c=c(q, \mathfrak{a}, \mathfrak{b})$. Notice that Lemma \ref{K-lemma-lower} is stated 
	for symmetric kernels with $c=c(q, \mathfrak{b})$, but for quasi-symmetric kernels, we use 
	  a symmetrized form of the kernel $G$, 
	  which yields the same estimate with a constant $c$ that additionally depends on $\mathfrak{a}$ (see Sec. \ref{background}). 

	If $u(x)\ge v(x)+ \G \mu(x)$ for $x \in \Omega$, then \eqref{K-lower-est-mu} is an immediate consequence of 
	 \eqref{v-lower-est-mu}. In particular, \eqref{K-lower-est-mu} holds $d \sigma$-a.e. 
	\end{proof}

	\section{Quasi-metric kernels}\label{q-m kernels}
	
	In this section, our main goal is to deduce the upper estimates \eqref{sublin-up} 
	of subsolutions associated with  equation  
	 \eqref{sublin-sigma-mu}, for quasi-metric kernels $G$. They match the lower estimates of supersolutions  \eqref{sublin-low}  obtained in Sec. \ref{lower bounds}.

	 We start with our main lemma. 			
	\begin{lemma}
		\label{qm_lemma-upper}
		Let $G$ be a quasi-metric  kernel on $\Omega \times \Omega$ with 
		quasi-metric constant $\kappa$. Let $0<q<1$ and $\nu, \sigma \in  \M(\Omega)$. 
			Then, for all $x \in \Omega$,  
					\begin{equation}\label{upper-est-qm} 
	 	\G [(\G \nu)^q d \sigma] (x)\le C   \left (\G \nu(x) \right)^q \, 
		\left[ \G\sigma(x) 
			+ \left(\K \sigma (x)\right)^{1-q} \right], 
	 	\end{equation}
		where $C=(2 \kappa)^{q}$.
	\end{lemma}
	
	\begin{proof}  Let $d \omega\defeq (\G \nu)^q d \sigma$. Then 
			\begin{align*}
		 {\G}\omega (x)= \int_0^\infty \frac{\omega(B(x, t))}{t^2} dt. 
		\end{align*}
  Clearly, 
 \begin{align*}
 \omega( B(x, t)) & = \int_{B(x, t)} (\G \nu)^q d \sigma  \\ & \le 
 \int_{B(x,t)}  (\G \nu_{B(x, 2 \kappa  t)})^q d \sigma 
 + 
 \int_{B(x,t)}  (\G \nu_{B(x, 2  \kappa  t)^c})^q d  \sigma\\  & \defeq I + II. 
 	\end{align*}
We estimate the first term, 
  \begin{align*}
 I = \int_{B(x, t)} (\G \nu_{   B(x, 2 \kappa  t)})^q d  \sigma \le    [\varkappa(B(x, t))]^q  \, [\nu(B(x, 2  \kappa  t))]^q. 
 	\end{align*}
	To estimate the second term, notice that 
	\begin{align*}
	\G \nu_{B(x, 2 \kappa t)^c}(y)  = \int_0^\infty \frac{\nu(B(y, r)\cap   B(x, 2 \kappa  t)^c)}{r^2} dr. 
\end{align*}
For all $y\in B(x, t)$ and $z\in B(y, r)$, we have 
$$
d(x, z) \le  \kappa \, [  d(x, y) +   d(y, z)]\le    \kappa \,  (t+r).
$$
Consequently, $  B(y, r)\subset   B(x, 2   \kappa \, t)$ if $0<r\le t$, so that 
$  B(y, r)\cap   B(x, 2   \kappa \, t)^c=\emptyset$. Moreover, 
$  B(y,r)\subset   B(x, 2   \kappa  r)$ if $r>t$.
Hence, for all $y\in   B(x, t)$,  
\begin{align*}
 {\G} \nu_{  B(x, 2  \kappa t)^c}(y) &  = \int_t^\infty \frac{\nu(  B(y, r)\cap   B(x ,2  \kappa t)^c)}{r^2} \, dr\\ & \le  \int_t^\infty \frac{\nu(  B(x, 2  \kappa \, r))}{r^2} \, dr
\\ & \le 2  \kappa \, \int_0^\infty \frac{\nu(  B(x, s))}{s^2} \, ds= 2  \kappa \, 
 {\G} \nu(x).
\end{align*}
It follows that 
\begin{align*}
 II = \int_{  B(x, t)}  ( {\G} \nu_{  B(x, 2   \kappa   t)^c})^q d   \sigma \le 
(2  \kappa)^q \left ( {\G} \nu(x) \right)^q \, 
   \sigma(  B(x, t )).
\end{align*}
Combining the preceding estimates, we obtain  
\begin{align*}
		 \omega(B(x, t))  \le   [\varkappa(B(x, t))]^q  \,[ \nu( B(x, 2 \kappa t))]^q +  (2  \kappa)^q \left ( {\G} \nu(x) \right)^q \, 
   \sigma(  B(x, t)).
		\end{align*}
		Hence, 
	\begin{align*}
		  {\G}\omega (x) & = \int_0^\infty \frac{\omega(  B(x, t))}{t^2} dt\\
		& \le \int_0^\infty \frac{  [\varkappa(  B(x, t))]^q  \, [\nu(  B(x, 2   \kappa t))]^q }{t^2} dt 
		\\ & + (2  \kappa)^q \, \left( {\G} \nu(x) \right)^q \,  \int_0^\infty \frac{   \sigma(  B(x, t))}{t^2} dt .
\end{align*}	
Using H\"{o}lder's  inequality in the first integral, we deduce 
\begin{align*}
		  {\G}\omega (x) & \le \left( \int_0^\infty \frac{[  \varkappa(B(x, t))]^{\frac{q }{1-q}}}{t^2} dt \right)^{1-q} \,  \left( \int_0^\infty \frac{\nu(B(x, 2   \kappa t))}{t^2} dt \right)^{q}
		\\ & + (2  \kappa)^q  \, \left( {\G} \nu(x) \right)^q \,  \int_0^\infty \frac{   \sigma(B(x, t))}{t^2} dt \\ & = (2  \kappa)^q  \,  \left(  {\K}   \sigma(x)\right)^{1-q}
		 \left(  {\G}\nu (x)\right)^{q} + (2  \kappa)^q  \, \left( {\G} \nu(x) \right)^q \,   {\G}  \sigma(x).
\end{align*}	
This proves estimate \eqref{upper-est-qm}. 
	\end{proof} 
	
	\begin{lemma}
		\label{qm_cor-upper}
		Let $G$ be a quasi-metric  kernel on $\Omega \times \Omega$ with 
		quasi-metric constant $\kappa$. 	Let $0<q<1$ and $ \mu, \sigma \in  \M(\Omega)$. 		Then any subsolution $u \ge 0$ such that  $u\le \G(u^q d \sigma)+\G \mu<+\infty$ $d \sigma$-a.e.,  satisfies the estimate
		\begin{equation}\label{upper-est-qm-cor} 
	 		u(x)\le C \,   \left [ \left(\G\sigma(x)\right)^{\frac{1}{1-q} } 
			+ \K \sigma (x) + \G\mu(x)\right], 
	 	\end{equation}
		for all $x\in \Omega$ such that $u(x)\le \G(u^q d \sigma)(x) +\G\mu(x)<+\infty$, where 
		$C= (8 \kappa)^{\frac{q}{1-q}}$. In particular, \eqref{upper-est-qm-cor} holds $d \sigma$-a.e. 
	\end{lemma}
	
		\begin{proof} Let $d \nu\defeq u^q d\sigma+ d\mu$, so that $u\le \G \nu$ $d \sigma$-a.e., 
		and consequently $d \nu \le  (\G \nu)^q d\sigma + d \mu$. Then 
		\[
		\G \nu(x)  \le \G [(\G \nu)^q d\sigma](x) + \G \mu(x), \quad \forall x \in \Omega. 
		\]
	 By Lemma \ref{qm_lemma-upper}, 
	\[
	\G [(\G \nu)^q d\sigma](x) \le (2 \kappa)^{q}  
		 \left (\G \nu(x) \right)^q \, 
		\left[ \G\sigma(x) 
			+ \left(\K \sigma (x)\right)^{1-q} \right]. 
	\]
Hence, 
		\begin{align*}
		 \G \nu(x) & \le \G [(\G \nu)^q d\sigma] + \G \mu(x) \\ & \le (2 \kappa)^{q}  
		 \left (\G \nu(x) \right)^q \, 
		\left[ \G\sigma(x) 
			+ \left(\K \sigma (x)\right)^{1-q} \right] +  \G \mu(x)\\ &\le (4 \kappa)^{q}  
		 \left (\G \nu(x) \right)^q \, 
		\left[ (\G\sigma(x) )^{\frac{1}{1-q}}
			+ \K \sigma (x) \right]^{1-q} +  \G \mu(x).
			\end{align*}	
	By Young's inequality, 
	\begin{align*}
	 (4 \kappa)^{q}  
		 & \left (\G \nu(x) \right)^q \, \left[ (\G\sigma(x) )^{\frac{1}{1-q}}
			+ \K \sigma (x) \right]^{1-q}\\
			& \le q \, \G \nu(x) 
			+ (1-q) \, (4 \kappa)^{\frac{q}{1-q}} \left[ (\G\sigma(x) )^{\frac{1}{1-q}}
			+ \K \sigma (x) \right]. 
	 \end{align*}
	Hence, 
		\begin{align*}
		 \G \nu(x)  \le q \,   \G \nu(x) + (1-q)  (4 \kappa)^{\frac{q}{1-q}}   \left [\left( \G \sigma(x) \right)^{\frac{1}{1-q}} +  \K  \sigma(x)  \right] +  \G \mu(x).
		 \end{align*}
		 
		 For $x \in \Omega$ such that $\G \nu(x)<+\infty$, we can move $q \,   \G \nu(x) $ to the left-hand side. 
		Then, dividing both sides by $1-q$, we obtain    
		\begin{align*}
		 \G \nu(x)  & \le  (4 \kappa)^{\frac{q}{1-q}}    \left [\left( \G \sigma(x) \right)^{\frac{1}{1-q}} +  \K  \sigma(x)  \right] + \frac{1}{1-q} \, \G \mu(x).
		 \end{align*} 
		Using the inequality $x+1\le e^{x}$ with 
	$x=\frac{q}{1-q}$ we estimate roughly $\frac{1}{1-q}\le e^{\frac{q}{1-q}} <(8\kappa)^{\frac{q}{1-q}}$, since 
	$2 \kappa\ge 1 >\frac{e}{4}$.  Hence, 
		 \begin{align*}
		 \G \nu(x)     \le (8 \kappa)^{\frac{q}{1-q}} \left [\left( \G \sigma(x) \right)^{\frac{1}{1-q}} +  \K  \sigma(x) + \G \mu(x) \right]. 
		 		 \end{align*}
	Thus, 
	\[
	u(x) \le  \G \nu(x)  \le (8 \kappa)^{\frac{q}{1-q}} \left [\left( \G \sigma(x) \right)^{\frac{1}{1-q}} +  \K  \sigma(x) + \G \mu(x) \right],   
	\] 
	for all $x \in \Omega$ such that  $\G \nu(x)=\G(u^q d \sigma)(x) +\G\mu(x)<+\infty$. This completes the proof of Lemma \ref{qm_cor-upper}. 
		\end{proof} 
			
	  In the remaining part  of this section, we rely on the fact  that a quasi-metric kernel with quasi-metric constant $\kappa$ obeys the (WMP) with constant 
	 $\mathfrak{b}=2\kappa$ by Lemma \ref{qmm-wmp} above.

	From the next lemma, we will deduce 	estimate	\eqref{upper-est-qm-cor}  for \textit{all} $x \in \Omega$, provided $u(x) \le \G(u^q d \sigma)(x)+\G \mu(x)$, including the case $u(x)=+\infty$.

	\begin{lemma}
		\label{qm-upper-K}
		Let $G$ be a quasi-metric  kernel on $\Omega \times \Omega$ with 
		quasi-metric constant $\kappa$. Let $0<q<1$ and $ \mu, \sigma \in  \M(\Omega)$. 
		 Then the function 
		 \begin{equation}\label{def-h} 
		  h(x)\defeq (\G\sigma(x))^{\frac{1}{1-q}} 
			+ \K \sigma (x) + \G \mu(x), \quad x \in \Omega,
		\end{equation} 
		 satisfies the estimate  
		 	\begin{equation}\label{upper-est-qm-K} 
	 	\G (h^q d \sigma) (x)\le C  \, h(x), \qquad \forall x \in \Omega, 
	 	\end{equation}
		 where $C$ is a constant which depends only on $q$ and $\kappa$.
	\end{lemma}
	
	\begin{proof}  
	Notice that by Lemma \ref{qm_lemma-upper} with $\nu=
	\mu$, we have
	\begin{align*}
	\G [(\G \mu)^q d \sigma] (x)& \le (2 \kappa)^q  \, (\G \mu(x))^q [\G \sigma(x) +(\K \sigma(x))^{1-q}], \\ 
	& \le  (4 \kappa)^q  h(x), \qquad  \forall x \in \Omega. 
	\end{align*}
	Therefore, it remains to prove \eqref{upper-est-qm-K} for $\mu=0$. 
	The proof  is similar to that of  Lemma \ref{qm_lemma-upper}, but 
	with nonlinear potentials $(\G\sigma)^{\frac{1}{1-q}}$ and $\K \sigma$ in place of the linear 
	potential $\G \nu$. Fix $x \in \Omega$, where without loss of generality we may assume that $h(x)<\infty$. In particular, $\K \sigma(x) < \infty$, and hence $\varkappa(B(x, s))<\infty$ for every $s>0$. 
	
	We have  
			\begin{align*}
		 \G(h^q d \sigma) (x)= \int_0^\infty \frac{ \int_{B(x, t)} (h(y))^q d \sigma(y)}{t^2} dt. 
		\end{align*}
		
		Notice that, for a constant $c(q)>0$ which depends only on $q$, 
		we have $h(y)\le c(q) (h_1(y)+h_2(y))$,  where
			\begin{align*} 
			h_1(y) &=[\G\sigma_{B(x, 2 \kappa t)}(y)]^{\frac{1}{1-q}} + \K \sigma_{B(x, 2 \kappa t)}(y)\\ h_2(y) &=[\G\sigma_{B(x, 2 \kappa t)^c}(y)]^{\frac{1}{1-q}} 
			+ \K \sigma_{B(x, 2 \kappa t)^c}(y) . 
	\end{align*}	
	Clearly, 
 \begin{align*}
 \int_{B(x, t)} (h(y))^q d \sigma(y) & \le c(q)^q  
 \int_{B(x, t)}  (h_1(y))^q d \sigma(y) \\ &+  c(q)^q  
 \int_{B(x, t)}  (h_2(y))^q d \sigma(y) \defeq I + II. 
 	\end{align*}
		
To estimate term $I$, notice that $\varkappa(B(x, 2 \kappa t))<\infty$. Hence,  
 by 
 \cite{QV2}*{Theorem 1.1 and Corollary 5.9} there exists a nontrivial 
solution $u_{B(x, 2 \kappa t)}\in L^q(\Omega, \sigma_{B(x, 2 \kappa t)})$  to the equation $u=\G ( u^q d\sigma_{B(x, 2 \kappa t)})$, and      
\[
\int_{B(x, 2 \kappa t)} [u_{B(x, 2 \kappa t)}(y)]^q d \sigma(y) \le  [\varkappa(B(x, 2 \kappa t))]^{\frac{q}{1-q}}.
\] 
We use the lower estimate 
\[
 u_{B(x, 2 \kappa t)} 
 \ge c(q, \kappa) \,   h_1 \,\quad d \sigma\textrm{-a.e.} \, \, \textrm{in} \, \, B(x, 2 \kappa t), 
\] 
which follows by combining \eqref{G-lower-est} 
and  \eqref{K-lower-est}. We estimate, 
\begin{align*} 
I = \int_{B(x,t)}  (h_1(y))^q d \sigma(y) \le C_1(q, \kappa) \, [\varkappa  (B(x, 2 \kappa t))]^{\frac{q}{1-q}},  
	\end{align*}
	for some   $C_1(q, \kappa)$ depending only on $q, \kappa$.  
	Integrating both sides over $(0, +\infty)$ 
with respect to $dt/t^2$, we deduce  
\begin{align*}
		 \G(h_1^q d \sigma) (x)& \le C_1(q, \kappa) \int_0^\infty  \frac{[\varkappa  (B(x, 2 \kappa t))]^{\frac{q}{1-q}}}{t^2} dt \\ & = 2\kappa \,C_1(q, \kappa) \, \K\sigma(x) \le 2\kappa \, C_1(q, \kappa) \, h(x). 
		\end{align*}

We next estimate term $II$ as in the proof of Lemma \ref{qm_lemma-upper}. 
If $y \in B(x,t)$, then 
by the quasi-triangle inequality we have 
\begin{align*}
B(y, s) \subset B(x, 2 \kappa t) \,\,\, \text{if} \, \, \, 0<s\le t, \qquad B(y, s) \subset B(x, 2 \kappa s)  \,\,\,  \text{if} \,\,  \, s>t.
\end{align*}
In particular, $B(y, s)\cap B(x, 2 \kappa t)^c=\emptyset$ for $0<s\le t$. 

Hence, using \eqref{varkappa-def} with $E=B(x, 2 \kappa t)^c$,   we obtain,  for all $y \in B(x,t)$, 
\begin{align*}
h_2(y)  &= \left[ \int_0^\infty \frac{\sigma (B(y, s)\cap B(x, 2 \kappa t)^c)}{s^2} ds \right]^{\frac{1}{1-q}} \\
& + \int_0^\infty \frac{[\varkappa (B(y, s)\cap B(x, 2 \kappa t)^c)]^{\frac{q}{1-q}}}{s^2} ds 
\\  & \le \left[ \int_t^\infty \frac{\sigma (B(x, 2 \kappa s))}{s^2} ds \right]^{\frac{1}{1-q}} + \int_t^\infty \frac{[\varkappa (B(x, 2 \kappa s))]^{\frac{q}{1-q}}}{s^2} ds \\ & \le c(q, \kappa) \,  h(x). 
\end{align*}
This estimate yields 
\begin{align*}
II  \le C_2(q, \kappa) \, (h(x))^q \, \sigma(B(x, t)),
	\end{align*}
where $C_2(q, \kappa) $ depends only on $q$ and $\kappa$. Integrating again both sides of   the preceding inequality over $(0, +\infty)$ 
with respect to $dt/t^2$, we see  that 
\begin{align*}
\G(h_2^q d \sigma)(x)\le C_2(q, \kappa) \, (h(x))^q \, \G \sigma(x) \le C_2(q, \kappa)  \, h(x),  
	\end{align*}
since obviously $\G \sigma(x)\le (h(x))^{1-q}$. 
Combining the above estimates  completes the proof of  \eqref{upper-est-qm-K} for all $x\in \Omega$ in the remaining case 
$\mu=0$. 
	\end{proof} 
		
		\begin{cor}
		\label{qm_cor-upper-K}
		Let $G$ be a quasi-metric  kernel on $\Omega \times \Omega$ with 
		quasi-metric constant $\kappa$. 		Let $0<q<1$ and $ \mu, \sigma \in  \M(\Omega)$. 	Then every subsolution $u$ for which $u\le \G(u^q d \sigma) + \G \mu<+\infty$ $d \sigma$-a.e.  satisfies the estimate
		\begin{equation}\label{upper-est-qm-cor-K} 
	 		u(x)\le (8 \kappa)^{\frac{q}{1-q}} \,   \left [ \left(\G\sigma(x)\right)^{\frac{1}{1-q} } 
			+ \K \sigma (x) + \G \mu(x) \right], 
	 	\end{equation}
		for all $x\in \Omega$ such that $u(x)\le \G(u^q d \sigma)(x)+ \G \mu(x)$.  In particular, 
		\eqref{upper-est-qm-cor-K} holds $d \sigma$-a.e.
	\end{cor}
	
	\begin{proof}  Let $h \defeq \left(\G\sigma \right)^{\frac{1}{1-q} } 
			+ \K \sigma + \G \mu$.   Fix $x \in \Omega$. In view of Lemma \ref{qm_cor-upper}, 
	$u(x) \le (8 \kappa)^{\frac{q}{1-q}} h(x)$	 provided $u(x) \le \G(u^q d \sigma)(x) + \G \mu(x)<+\infty$. 
			Therefore, it only remains to show  that $h(x) =+\infty$  
	whenever $\G(u^q d \sigma)(x) + \G \mu(x)=+\infty$. This   is obvious if $\G \mu(x)=+\infty$. 
	
	Suppose $\G \mu(x)<+\infty$, but $\G(u^q d \sigma)(x) =+\infty$. Since $u \le (8 \kappa)^{\frac{q}{1-q}} \, h$ 
	$d\sigma$-a.e.,  we have 
	\[
	\G(u^q d \sigma)(x)\le (8 \kappa)^{\frac{q^2}{1-q}} \, \G(h^q d \sigma)(x). 
	\]
	By  Lemma \ref{qm-upper-K}, $\G(h^q d \sigma)(x) \le C(q, \kappa) \, h(x)$ for all $x\in\Omega$. 
	Hence,  $\G(u^q d \sigma)(x) =+\infty\Longrightarrow h(x)= +\infty$. \end{proof}

		\begin{lemma}\label{exist-lemma} Let $\mu, \sigma \in \mathcal{M}^{+}(\Omega)$ ($\sigma\not=0$) and $0<q<1$.  Suppose $G$ is a  quasi-metric kernel on $\Omega\times \Omega$. Then a nontrivial solution $u$ to \eqref{sublin-sigma-mu}  exists if and only if 
		 $\G \sigma <+\infty$, $\K \sigma <+\infty$, and $\G \mu < +\infty$ 
		 $d\sigma$-a.e., and satisfies  the bilateral pointwise estimates  
				\begin{equation}
				\label{sublin-low-e} 
	 c \,  [ (\G \sigma(x))^{\frac{1}{1-q}} + 	 \K\sigma (x)]+  \G \mu(x) \le 	u (x),
	 \end{equation}	
	 and 
	 		\begin{equation}\label{sublin-up-e} 
		u(x)  \le  C \, [ (\G \sigma(x))^{\frac{1}{1-q}} + 	 \K\sigma (x) +  \G \mu(x) ],   
					 \end{equation}	
		$d \sigma$-a.e. in $\Omega$, 
		where $c, C$ are positive constants which depend only on $q$ and the 
		quasi-metric constant $\kappa$ of the kernel $G$. 	
	\end{lemma}

	\begin{proof} Let $u_0 = c_0 \, \left[ (\G \sigma)^{\frac{1}{1-q}} + \G \mu\right]$, where a small constant $0< c_0 \le 1$   is to be  determined later. 
		We claim that  $u_0$ is a positive subsolution to \eqref{int-sup}, i.e., 
		$u_0 \le \G(u_0^q d \sigma) + \G\mu$. It suffices to verify the inequality 
		\begin{equation}\label{gv-est-a} 
		c_0  \, \left[ (\G \sigma)^{\frac{1}{1-q}}  + 
		\G \mu\right] \le c_0^q \, \G \left[ (\G \sigma)^{\frac{q}{1-q}} d \sigma\right]+ \G \mu.
		\end{equation}
		
		By \cite{GV}*{Lemma 2.5 and Remark 2.6} with $r=\frac{1}{1-q}$, we have 
		\begin{equation}
(\G \sigma)^{\frac{1}{1-q}}  \le \frac{1}{1-q}  \mathfrak{b}^{\frac{q}{1-q} }\, \G \left[ (\G \sigma)^{\frac{q}{1-q}} d \sigma\right],
\end{equation}
where 	$\mathfrak{b}$ is the constant in the (WMP) for $G$. Hence, \eqref{gv-est-a} 
holds if 
we pick $c_0 $ small enough so that $c_0^{q-1} \ge  \frac{1}{1-q} \mathfrak{b}^{\frac{q}{1-q} }$, which proves the claim.   

	We next define the sequence $\{ u_j \}_{j=0}^{\infty}$  by   
			\[ u_{j+1} \defeq \G(u_j^q d \sigma) +\G \mu, \quad j=0, 1, 2, \ldots .\] 
		Since $u_0$ is a subsolution, we have $u_1\ge u_0$. By induction, we see that 
		$\{u_j\}$ is non-decreasing, so that  $u_j\uparrow u$, and 
		\[ u_j\le u_{j+1} \defeq \G (u_j^q d \sigma) +\G \mu, \quad j=0, 1, 2, \ldots .\] 

To show that each $u_j$ is a subsolution,  arguing by  induction and applying Lemma \ref{qm-upper-K}, 
		we estimate 
		\[
		u_j \le C \, \,[ (\G \sigma)^{\frac{1}{1-q}} + 	 \K\sigma +  \G \mu ]< +\infty 
\quad d \sigma\textrm{-a.e.} 
		\]
		where $C=C(j, q, \kappa)$ may depend on $j$. 
	Then by Corollary \ref{qm_cor-upper-K},		the preceding inequality holds with 
	a constant $C=C(q, \kappa)$ which does not depend on $j$. 
		By the monotone convergence theorem, 
		$u<+\infty$ $d\sigma$-a.e. is a nontrivial solution to \eqref{int-sup}, which 
		satisfies \eqref{sublin-up-e}. The lower estimate  \eqref{sublin-low-e} holds for any (super) solution $u$ by 
		Corollary \ref{cor-lower-est}, since $G$ is a quasi-metric kernel which satisfies the (WMP) 
		by Lemma \ref{qmm-wmp}.  
			\end{proof} 
			

	\begin{lemma}\label{equiv-lemma} Let $\mu, \sigma \in \mathcal{M}^{+}(\Omega)$ ($\sigma\not=0$) and $0<q<1$.  Suppose $G$ is a  quasi-metric kernel on $\Omega\times \Omega$. Then the following conditions 
	are equivalent:
	
(i) 	 $\G \sigma <+\infty$, $\K \sigma <+\infty$, and $\G \mu < +\infty$ 
		 $d\sigma$-a.e. 
		 
 (ii) $\G \sigma\not\equiv +\infty$, $\K \sigma \not\equiv +\infty$, and $\G \mu \not\equiv +\infty$. 
		 
(iii) Conditions \eqref{cond-a}--\eqref{cond-c} hold for some (or, equivalently, all) $x_0\in \Omega$ and 
$a>0$. 
	\end{lemma} 
	
	\begin{proof} We first show that, if any one of conditions \eqref{cond-a}--\eqref{cond-c} holds for some  $x_0\in \Omega$ and 
$a>0$, then it holds 
for all $x_0\in \Omega$ and 
$a>0$. Let $x \in \Omega$. Since $G(x, x_0)>0$,  it follows that $x \in B(x_0, R)$ for 
	 $R$ large, where we may assume   $R>a$. Then, by the quasi-triangle 
	inequality, $B(x, t) \subset B(x_0, 2 \kappa  t)$, for all $t\ge R$. Consequently,  for all $t\ge R$, we have 
		\begin{align*}
	& \sigma(B(x, t)) \le \sigma (B(x_0, 2 \kappa  t)), \, \, \\&  
	 \varkappa(B(x, t))\le  \varkappa(B(x_0, 2 \kappa  t)), \, \, \\& 
	 \mu(B(x, t)) \le \mu (B(x_0, 2 \kappa  t)).
		 \end{align*}
	It follows that 
	\begin{align}
	& \int_R^\infty \frac{\sigma(B(x, t))}{t^2} \, d t  \le \int_R^\infty \frac{\sigma(B(x_0, 2 \kappa t))}{t^2} \, d t < \infty,  \label{cond-aa} 
	\\ 
	& \int_R^\infty \frac{\left[{\varkappa(B(x, t))}\right]^{\frac{q}{1-q}}}  {t^2}  \le
	 \int_R^\infty \frac{\left[{\varkappa(B(x_0, 2 \kappa  t))}\right]^{\frac{q}{1-q}}}  {t^2}\,
	< \infty,
	\label{cond-bb}
	 \\  
	& \int_R^\infty \frac{\mu(B(x, t))}{t^2} \, d t \le 
	\int_R^\infty \frac{\mu(B(x_0, 2 \kappa  t))}{t^2} \, d t < \infty,   
				 \label{cond-cc}  
				 \end{align}	 
respectively, since the substitution $s= 2 \kappa  t$ in the integrals on the right-hand side 
shows that $s > 2 \kappa \, R\ge R>a$. 

It follows from \eqref{cond-aa}--\eqref{cond-cc} that  $\sigma(B(x, t))< \infty$, $\varkappa(B(x, t)) < \infty$ 
and $\mu(B(x, t))< \infty$ for all $t>R$, and 
hence for all $t>0$. We deduce that, for all $x\in \Omega$ and $a>0$, we have (independently 
for each of the following integrals) 
\begin{align}
	&  \int_a^\infty \frac{\sigma(B(x, t))}{t^2} \, d t  < \infty,  \label{cond-aaa} \\ 
	&  \int_a^\infty \frac{\left[{\varkappa(B(x, t))}\right]^{\frac{q}{1-q}}}  {t^2} dt <\infty,
	\label{cond-bbb} \\  & \int_a^\infty \frac{\mu(B(x, t))}{t^2} \, d t  < \infty.   
				 \label{cond-ccc}  
				 \end{align}	 
	Thus, if condition (iii) of Lemma \ref{equiv-lemma} holds for some $x_0$ and $a$, 
	then it holds for all  $x_0 \in \Omega$ and $a>0$. 
	
	We notice that, obviously, 
	(i)$\Longrightarrow$(ii)$\Longrightarrow$(iii).  Let us show 
	 (iii)$\Longrightarrow$(i). As was noticed 
	above, it follows from \eqref{cond-bbb} that  $\varkappa(B(x, t)) < \infty$ for all $x \in \Omega$ and  $t>0$. Hence, for any finite measure $\nu \in \M(\Omega)$, we have 
	\[
	\int_{B(x, t)} (\G \nu)^q d \sigma \le [\varkappa(B(x, t) )]^{q} \, 
	\Vert \nu\Vert^q<\infty.
	\]
	Applying the preceding inequality to $\nu=\sigma_{B(x, 2 \kappa  t)}$ and 
	$\nu=\mu_{B(x, 2 \kappa  t)}$, we deduce, 
	for all $x \in \Omega$ and $t>0$,  
	\[
	\int_{B(x, t)} (\G \sigma_{B(x, 2 \kappa t)})^q d \sigma <\infty, \quad \int_{B(x, t)} (\G \mu_{B(x, 2 \kappa  t)})^q d \sigma <\infty.
	\]		
	It follows that $\G \sigma_{B(x, 2 \kappa  t)}<\infty$ and $\G \sigma_{B(x, 2 \kappa t)}<\infty$ 
	$d \sigma$-a.e. in $B(x, t)$. 
	As was shown above in the proof of Lemma \ref{qm-upper-K} (the estimate of term II), we have, for all $y \in 	B(x, t)$, 
	\begin{align*}
	\G \sigma_{B(x, 2 \kappa t)^c}(y) & \le C \, \int_t^\infty \frac{\sigma(B(x, 2 \kappa  s))}{s^2} \, d s < \infty, \\ \G \mu_{B(x, 2 \kappa  t)^c}(y) & \le C \, \int_t^\infty \frac{\mu(B(x, 2 \kappa  s))}{s^2} \, d s < \infty. 
	\end{align*}
	Therefore, $\G \sigma<\infty$ and $\G \mu<\infty$ $d \sigma$-a.e. in all quasi-metric balls $B(x, t)$, and consequently in $\Omega=\cup_{t>0} B(x, t)$. 
	
	To verify that $\K \sigma<\infty$ $d \sigma$-a.e. in all quasi-metric balls $B(x, t)$, notice that $\varkappa(B(x, 2 \kappa t) )<\infty$, and hence  by Gagliardo's lemma (see \cite{G}*{Lemma 3.I} or \cite{QV2}*{Theorem 1.1}), there exists a nontrivial solution $u_{B(x, 2 \kappa  t)}\in L^q(\Omega, 
	\sigma_{B(x, 2 \kappa t)})$ to the equation $u=\G(u^q \sigma_{B(x, 2 \kappa t)} )$ understood $d \sigma_{B(x, 2 \kappa t)}$-a.e. 
	By  Lemma \ref{K-lemma-lower}, 
	\[
	\K \sigma_{B(x, 2 \kappa t)} \le  c \, u_{B(x, 2 \kappa t)}<\infty \quad d \sigma\textrm{-a.e.} \, \, \textrm{in} \, \, B(x, t).
	\]
Moreover, 
	as in the proof of Lemma \ref{qm-upper-K}, 
	for any $y\in B(x, t)$, we have 
	\[
	\K \sigma_{B(x, 2 \kappa  t)^c} (y) \le C \, 
	\int_t^\infty \frac{\left[{\varkappa(B(x, 2 \kappa  s))}\right]^{\frac{q}{1-q}}}  {s^2} \, d s  <\infty.  
	\]
	Thus, 		$\K \sigma<\infty$ $d \sigma$-a.e. in all quasi-metric balls  $B(x, t)$, and consequently 
in $\Omega$. This completes the proof of (iii)$\Longrightarrow$(i).  \end{proof}


	\section{Quasi-metrically modifiable kernels}\label{q-m-m kernels}
	
	In this section, we give bilateral pointwise estimates of solutions to sublinear integral equations 
	\eqref{sublin-sigma-mu} for quasi-metrically modifiable kernels $G$ with modifier $m$. 			
	We recall that the kernel $G$ is quasi-metrically modifiable, with constant $\kappa$, if for 
	 a  positive  function $m\in C(\Omega)$ called a modifier, 
		\begin{equation}\label{g-mod} 
			\widetilde{G}(x, y) \defeq \frac{G(x, y)}{m(x) \, m(y)}, \qquad x, y \in \Omega, 
		\end{equation}
		is a quasi-metric kernel with quasi-metric constant $\kappa$.

	Examples of  quasi-metrically modifiable kernels can be found in \cite{An}, \cite{FNV}, \cite{FV2}, \cite{H}. In particular, for bounded domains $\Omega\subset\R^n$ satisfying the boundary Harnack principle, such as bounded Lipschitz, NTA or uniform domains, Green's kernels $G$ for the Laplacian and fractional Laplacian (in the case $0<\alpha\le 2$, $\alpha<n$) are quasi-metrically modifiable.

	Let $0<q<1$ and $\mu, \sigma \in \M(\Omega)$. 
	We discuss relations between solutions (as well as sub- and super-solutions) to the equations
	\begin{align}
	\label{int-eq-u}
			u & = \G(u^q d \sigma) + \G\mu, \quad 0< u<\infty  \quad d \sigma\textrm{-a.e.} \, \, \text{in} \, \, \Omega, 
				\\
		\label{int-eq-v}
			v & = \widetilde{\G} (v^q d \tilde{\sigma}) + \widetilde{\G} \tilde{\mu}, \quad 0< v<\infty   \quad d \tilde{\sigma}\textrm{-a.e.} \, \, \text{in} \, \, \Omega,
	\end{align}
	where $\widetilde{G}$ is the modified kernel \eqref{g-mod} with modifier $m$,
	  and 
	  \begin{align}\label{v-u}
	  v \defeq \frac{u}{m}, \quad d \tilde{\sigma} \defeq m^{1+q} \, d\sigma, \quad d \tilde{\mu}= m \, d \mu.
	  \end{align}
	  Clearly, equations \eqref{int-eq-u} and \eqref{int-eq-v} are equivalent.
	  	
	We next consider a pair of $(1, q)$-weighted norm inequalities with the same least constant 
	$\tilde{\varkappa}=\tilde{\varkappa}(\Omega)$, 
	\begin{align}
		\label{weighted-G}
		\Vert \G \nu \Vert_{L^q(m \, d\sigma)} & \le \tilde{\varkappa} \, \int_\Omega m \, d\nu,  \quad \forall \,  \nu \in \M(\Omega), 
	\\
		\label{weighted-G-t}
		\Vert \widetilde{\G} \nu \, \Vert_{L^q(\Omega, \tilde{\sigma})} & \le \tilde{\varkappa} \, \Vert \nu \Vert,  \quad \forall \,  \nu \in \M(\Omega).
	\end{align}
	Here again \eqref{weighted-G-t} is simply an equivalent restatement of \eqref{weighted-G}  in terms of  $\widetilde{\G}$, $\tilde{\sigma}$ 
	instead of $\G$, $\sigma$.

	For a Borel set  $E \subset \Omega$,  
	we denote by $\tilde{\varkappa}(E)$ the least constant in the localized  versions 
	of  inequalities  \eqref{weighted-G}, \eqref{weighted-G-t} with $\sigma_E$ in place of $\sigma$.

	Let  $\widetilde{B}(x, r)$  be a quasi-metric ball associated with the quasi-metric 
	$\tilde{d}=1/\widetilde{G}$, 		
	\begin{equation}\label{def-tilde-B} 
	\widetilde{B}(x, r) \defeq \left\{ y \in \Omega\colon \, \, \widetilde{G}(x, y)>\frac{1}{r} \right\}, \quad  
	x \in \Omega, \, \, r>0.
		\end{equation}
		
		We recall that the modified intrinsic potential $\widetilde{\K} \sigma$ is defined by 
		\begin{equation}\label{def-tilde-K} 
	\widetilde{\K} \sigma  (x) \defeq \int_0^\infty \frac{ [\tilde{\varkappa}(\widetilde{B}(x, r))]^{\frac{q}{1-q}} }{r^2} \, dr, \qquad x\in \Omega.
		\end{equation}
		
	\begin{lemma}
		\label{qm_lemma-lower}
		Let $G$ be a (QS)   kernel on $\Omega \times \Omega$ 
		such that the modified kernel $\widetilde{G}$, defined by \eqref{g-mod} with  modifier $m$,  
		satisfies the (WMP).  Then any nontrivial supersolution $u$ to \eqref{int-eq-u}  satisfies 
		the estimate 
		\begin{equation}\label{lower-est-qm} 
	 		u\ge c \, m \, \left( \left [ \frac{\G( m^q d \sigma)}{m}\right]^{\frac{1}{1-q} }
			+ \widetilde{\K} \sigma\right)     +  \, \G\mu\qquad d \sigma\textrm{-a.e.},
	 	\end{equation}
		where $c$ is a positive constant which depends only on $q$, the quasi-symmetry constant 
		$\mathfrak{a}$, and  the constant $\mathfrak{b}$ in the  (WMP) for $\widetilde{G}$.
	\end{lemma}
		\begin{proof}  Let $v$, $\tilde{\sigma}$, and $\tilde \mu$ be defined by \eqref{v-u}. Then 
		$v>0$ $d \tilde\sigma$-a.e.  is a nontrivial supersolution, i.e.,  
	$\widetilde{\G}(v^q d \tilde \sigma)+ \widetilde{\G} \tilde{\mu}\le v<\infty$ $d \tilde\sigma$-a.e.  
	Since $\widetilde{G}$ is a quasi-metric kernel, 
	by Corollary \ref{cor-lower-est} 
	applied to $v$, $\tilde \sigma$, $\tilde \mu$ and $\widetilde{\G}$, we obtain 
$$
v \ge  c \left ( ( \widetilde{\G} \tilde{\sigma})^{\frac{1}{1-q}} +   \widetilde{\K} \sigma \right)+   \widetilde{\G} \tilde{\mu}
 \qquad d \tilde \sigma\textrm{-a.e.}
$$
This gives  estimate \eqref{lower-est-qm} for the supersolution $u$. 
\end{proof}

	The next lemma provides a matching  upper pointwise bound for subsolutions to sublinear integral equations, for quasi-metrically modifiable kernels  	$G$.

	\begin{lemma}
		\label{qm_lemma-upper-mod}
		Let $G$ be a quasi-metrically modifiable  kernel on $\Omega \times \Omega$ with   modifier $m$.
			Then any subsolution $u$ to \eqref{int-eq-u}   satisfies the estimate
		\begin{equation}\label{upper-est-qmm} 
	 		u\le C \, m \, \left( \left [ \frac{\G( m^q d \sigma)}{m}\right]^{\frac{1}{1-q} } 
			+ \widetilde{\K}  \sigma\right) + C \, \G\mu \quad d \sigma\textrm{-a.e.},
	 	\end{equation}
		where $C=C(q, \kappa)$, and $\kappa$ is the quasi-metric constant   for the modified kernel 
		$ \widetilde{G}$.
	\end{lemma}

	\begin{proof}  Let $v$, $\tilde{\sigma}$, and $\tilde \mu$ be defined by \eqref{v-u}. Then $v\ge 0$ is a subsolution, i.e.,  
	$v \le \widetilde{\G}(v^q d \tilde \sigma)+ \widetilde{\G} \tilde{\mu}$.  
	Since $\widetilde{G}$ is a quasi-metric kernel, 
	by Lemma \ref{qm_cor-upper} applied to $v$, $\tilde \sigma$, $\tilde \mu$ and $\widetilde{\G}$, we obtain 
$$
v \le  C \left ( ( \widetilde{\G} \tilde{\sigma})^{\frac{1}{1-q}} +   \widetilde{\K} \sigma +   \widetilde{\G} \tilde{\mu}
\right).
$$
This translates to  estimate \eqref{upper-est-qmm} for the subsolution $u$. 
	\end{proof}

	In the remaining part of this section we consider the modifiers $m=g$ given by 
	\begin{equation}\label{g-mod-def}
	g(x) =g^{x_0} (x) \defeq \min\{1, G(x, x_0)\}, \quad x \in \Omega,
		\end{equation}
	where $x_0\in \Omega$ is a fixed pole.

	The following lemma was proved in \cite{QV2}*{Lemma 3.4}   
		in a slightly weaker form, namely for the set $\Omega \setminus \{ x\colon \, G(x,x_0) = +\infty \}$ in place of $\Omega$ (see also \cite{HN}*{Sec. 8}).  Its proof is based on the so-called 
	 Ptolemy's inequality for quasi-metric spaces \cite{FNV}*{Sec. 3}, which states that, if $d$ is a quasi-metric with constant $\kappa$ on $\Omega$, then 
		\begin{equation}\label{ptolemy} 
	 		d(x,y) d(x_0, z) \le 4 \kappa^2 \Big[d(x, z) d(y, x_0) + d(x_0, x) d(z, y)\Big], 
	 	\end{equation}
		for all $x_0, x, y, z \in \Omega$.

	\begin{lemma}\label{mod-lemma} Let  $G$ be a quasi-metric kernel on $\Omega\times \Omega$ with quasi-metric constant $\kappa$.  Let $x_0\in \Omega$, and let $g(x) = \min\{ 1, G(x, x_0)\}$. 
		Then 
		\begin{equation}\label{ker-mod-G}
		\widetilde{G}(x,y) = \frac{G(x,y)}{g(x) g(y)}
			\end{equation}  
		is a quasi-metric kernel on $\Omega\times \Omega$ 
	 with quasi-metric constant 	$4 \kappa^2$. In particular,  $\widetilde{G}$ 
satisfies the (WMP) with constant $\mathfrak{b}=8 \kappa^2$. 
	\end{lemma}
	
	\begin{proof} Fix a pole $x_0\in \Omega$. We need  to prove the inequality 
	\begin{equation}\label{goal-in} 
	 		\frac{g(x) g(y)}{G(x,y)} \le 4 \kappa^2 \Big[\frac{g(x) g(z)}{G(x,z)} + \frac{g(y) g(z)}{G(z,y)}\Big], 
	 	\end{equation}
		for all $x, y, z \in \Omega$.  
		
	Since $G$ is a quasi-metric kernel, multiplying both sides of  \eqref{quasitr} by $g(x) g(y)$,  we deduce 
	\begin{equation}\label{goal-tr} 
	 		\frac{g(x) g(y)}{G(x,y)} \le \kappa \,  \Big[\frac{g(x) g(y)}{G(x,z)} + \frac{g(x) g(y)}{G(z,y)}\Big], 
	 	\end{equation}
		for all $x, y, z \in \Omega$.  
		
			In the case $1\le G(z, x_0) \le +\infty$, which ensures that $g(z)=1$,   we see that \eqref{goal-in} is immediate from \eqref{goal-tr}, since obviously $g(y)\le g(z)$ and $g(x)\le g(z)$, and $\kappa \ge \frac{1}{2}$.

In the remaining case  $G(z, x_0)<1$, we have $g(z) =  G(z, x_0)$.  Obviously, $g(y) \le G(y, x_0)$ and $g(x) \le G(x, x_0)$. Hence, letting $d\defeq 1/G$ in \eqref{ptolemy}   and multiplying both of its sides  by $g(x) g(y) \, G(z, x_0)$, we obtain  
		\begin{equation}\label{goal-pt} 
	 		\frac{g(x) g(y)}{G(x,y)} 
			\le 4 \kappa^2 \Big[\frac{g(x)}{G(x, z)} + \frac{g(y)}{G(z, y)}\Big] \, G(z, x_0),  
	 	\end{equation}
	which coincides with \eqref{goal-in} in this case. 	Thus,  \eqref{goal-in} 
		follows from \eqref{goal-pt}. This proves that $\widetilde{G}$ is a quasi-metric kernel with quasi-metric constant 
	$\tilde\kappa= 4 \kappa^2$. Then $\widetilde{G}$ satisfies the 
	(WMP) with constant $\mathfrak{b}=2 \tilde\kappa=8 \kappa^2$ by Lemma \ref{qmm-wmp}.	
	\end{proof}

	 In the next lemma, we give a criterion for the existence of (super) solutions 
	in the case of quasi-metrically modified kernels $G$.  
	
	\begin{lemma}\label{equiv-lemma-mod} Let $\mu, \sigma \in \mathcal{M}^{+}(\Omega)$ ($\sigma\not=0$) and $0<q<1$.  Suppose $G$ is a  quasi-metrically modifiable kernel on $\Omega\times \Omega$ with modifier  $m=g \in C(\Omega)$ defined by \eqref{g-mod-def}. Then there exists a nontrivial (super) solution to equation \eqref{int-eq-u} 
if and only if conditions \eqref{tilde-kappa} hold,  i.e., 
\begin{equation}\label{tilde-kappa-5}
\int_\Omega g \, d \mu<\infty \quad {\rm and} \quad  \tilde{\varkappa}(\Omega)<\infty,  
\end{equation}
where  $\tilde{\varkappa}(\Omega)$ denotes the least constant in the weighted norm inequality 
\begin{equation}\label{tilde-kap-1}
\Vert \G \nu\Vert_{L^q(\Omega, g d\sigma)}\le \tilde{\varkappa}(\Omega) \, \int_\Omega g \, d \nu, \qquad 
\forall \nu \in \M(\Omega).
\end{equation}
	\end{lemma} 
	
	\begin{proof} Let $\widetilde{G}$ be the modified kernel defined by \eqref{ker-mod-G} and $\widetilde{B}(x, t)$ the corresponding quasi-metric ball. Let $m=g$ with 
	a fixed pole $x_0\in \Omega$. Obviously,   $g(x)\le G(x, x_0)$ and $g(x_0)\le 1$. It follows that 
\begin{equation}\label{tilde-g-0}
\widetilde{G}(x, x_0)=\frac{G(x, x_0)}{g(x) \, g(x_0)}\ge 1, \qquad \forall x\in \Omega. 
\end{equation}
Hence,   
	\begin{equation}\label{tilde-g-2}
	\tilde{g}(x)\defeq\min\{1, \widetilde{G}(x, x_0)\}\equiv 1, \qquad \forall x\in \Omega.
		\end{equation}

We set $d \tilde{\sigma}= g^{1+q} d \sigma$ and $d \tilde{\mu}= g \, d \mu$. By Lemma 
\ref{exist-lemma} and  Lemma \ref{equiv-lemma} 
with $\widetilde{G}$, $\tilde{\sigma}$ and $\tilde{\mu}$ in place of $G$, $\sigma$ and $\mu$, 
respectively, we see that there exists a nontrivial solution $v$ to  equation \eqref{int-eq-v} if and only if 
\[
\int_1^\infty \frac{\tilde{\sigma}(\widetilde{B}(x_0, t))}{t^2} d t +
	  \int_1^\infty \frac{[{\tilde{\varkappa}(\widetilde{B}(x_0, t))}]^{\frac{q}{1-q}}}  {t^2}dt+ \int_1^\infty \frac{\tilde{\mu}(\widetilde{B}(x_0, t))}{t^2} d t < \infty.     
				 \]	 

It follows from  \eqref{tilde-g-0} that $\widetilde{B}(x_0, t)=\Omega$ if $t>1$. Hence, the preceding condition is equivalent to 
\begin{equation}\label{tilde-g-3}
\tilde{\sigma}(\Omega)< \infty, \quad \tilde{\varkappa}(\Omega)< \infty, \quad \tilde{\mu}(\Omega)< \infty,  
\end{equation}
where $\tilde{\varkappa} (\Omega)$ is the least constant in the inequality 
\[
\Vert \widetilde{\G} \tilde{\nu}\Vert_{L^q(\Omega, \tilde g d\tilde \sigma)}\le \tilde{\varkappa}(\Omega) \, \int_\Omega \tilde g \, d  \tilde{\nu}, \quad 
\forall   \tilde{\nu} \in \M(\Omega).
\]
But $\tilde g\equiv 1$ in $\Omega$ by \eqref{tilde-g-2}. Hence, $\tilde{\varkappa}(\Omega)$ coincides 
with the least constant  in the inequality 
\begin{align}\label{tilde-nu-2}
\Vert \widetilde{\G} \tilde{\nu}\Vert_{L^q (\Omega, \tilde{\sigma})}\le \tilde{\varkappa}(\Omega) \, \Vert  \tilde{\nu}\Vert, 
\qquad \forall \tilde{\nu}\in \M(\Omega),
\end{align}
or, equivalently, \eqref{tilde-kap-1}. Letting $ \tilde{\nu}=\delta_{x_0}$ in \eqref{tilde-nu-2}, and noticing 
 that $\widetilde{\G} \tilde{\nu}(x)=\widetilde{G}(x, x_0)\ge 1$ in $\Omega$ by \eqref{tilde-g-2}, we obtain  
\[
[\tilde{\sigma}(\Omega)]^{\frac{1}{q}} \le \tilde{\varkappa}(\Omega). 
\]
Hence, the condition $\tilde{\sigma}(\Omega)<\infty$ in \eqref{tilde-g-3} is redundant. Since 
 $\tilde{\mu}(\Omega)   = \int_\Omega g \, d \mu$, 
it follows that \eqref{tilde-g-3} is equivalent to \eqref{tilde-kappa-5}. Thus, a nontrivial (super) solution $v$ to  \eqref{int-eq-v} exists if and only if \eqref{tilde-kappa-5} holds. Using  the relation \eqref{v-u} 
which translates \eqref{int-eq-v} to \eqref{int-eq-u}, we see that condition \eqref{tilde-kappa-5} is 
 necessary and sufficient for the existence of a nontrivial (super) solution $u$ to  \eqref{int-eq-u}.   
\end{proof}
	
	 In the following corollary, we give an alternative criterion for the existence of (super) solutions to \eqref{int-eq-u} in the case of quasi-metric kernels $G$, which complements Lemma \ref{equiv-lemma} above.

	\begin{cor}\label{equiv-cor} Let $\mu, \sigma \in \mathcal{M}^{+}(\Omega)$ ($\sigma\not=0$) and $0<q<1$.  Suppose $G$ is a  quasi-metric kernel on $\Omega\times \Omega$ such that $g \in C(\Omega)$, where  $g$ is defined by \eqref{g-mod-def}. Then there exists a nontrivial (super) solution to equation \eqref{int-eq-u} 
if and only if  \eqref{tilde-kappa-5} holds.
	\end{cor} 
	
	\begin{proof} By Lemma \ref{mod-lemma}, the kernel $G$ is quasi-metrically modifiable with modifier $m=g$. Hence, by Lemma \ref{equiv-lemma-mod} there exists a nontrivial (super) solution to equation \eqref{int-eq-u} 
if and only if  \eqref{tilde-kappa-5} holds. 
\end{proof}


	\section{Proofs of Theorem \ref{strong-thm}, Theorem \ref{strong-thm-mod}, and 
	Theorem \ref{main-thm-f}}\label{sec. 6}	

	    \begin{proof}[Proof of Theorem \ref{strong-thm}]
	The lower bound \eqref{sublin-low} for nontrivial supersolutions  follows from Corollary \ref{cor-lower-est}, for all (QS)\&(WMP) kernels $G$ with  $c=c(q, \mathfrak{a}, \mathfrak{b})$. In particular, it holds with $c=c(q, \kappa)$ 
	for quasi-metric kernels, which satisfy the (WMP) with constant $\mathfrak{b}=2 \kappa$  
	by Lemma \ref{qmm-wmp}.   
	
	The upper bound for subsolutions \eqref{sublin-up} is a consequence of 
	Corollary \ref{qm_cor-upper-K} for quasi-metric  kernels $G$ with $C=(8 \kappa)^{\frac{q}{1-q}}$. This proves statement (ii) of Theorem \ref{strong-thm}.  
	
	Combining the upper and lower bounds for sub and supersolutions, respectively, yields bilateral 
	estimates for solutions. The uniqueness of solutions is proved in Theorem \ref{main-thm-f}  below 
	for more general Borel measurable data $f\ge 0$ in place of $\G\mu$. This yields statement (i).

	Finally, the existence criterion in statement (iii)  is a consequence of  Lemma \ref{exist-lemma} 
	and Lemma \ref{equiv-lemma}. \end{proof}
	
	 \begin{proof}[Proof of Theorem \ref{strong-thm-mod}] Suppose that the 
	 kernel $G$ is quasi-metrically modifiable. Then the modified kernel $\widetilde{G}$ is quasi-metric,  and  obeys the (WMP) by Lemma \ref{qmm-wmp}. 
	 Hence, 
	 the lower bound for supersolutions follows from Lemma \ref{qm_lemma-lower}. The upper 
	 bound for subsolutions is a consequence of 
	 Lemma \ref{qm_lemma-upper-mod}. Combining these estimates gives the bilateral estimates of solutions. The uniqueness property of solutions is a consequence of the bilateral estimates, as shown below in the proof of Theorem \ref{main-thm-f}. This completes the proof of 
	 Theorem \ref{strong-thm-mod}. 
	 \end{proof}
	 
	 The existence criteria mentioned in Remarks 1 and 2 after Theorem \ref{strong-thm-mod} follow from  Lemma \ref{exist-lemma} 
	and Lemma \ref{equiv-lemma} 
	(applied to $\widetilde{G}$, 
	$d \tilde{\sigma}= m^{1+q} d \sigma$ and $d \tilde{\mu}= m \, d \mu$ in place 
	of $G$, $\sigma$ and $\mu$, respectively) and 
	 Lemma \ref{equiv-lemma-mod}.
\smallskip

In conclusion, we extend our results to equation \eqref{sublin-eq-f}
considered at the beginning of the Introduction, 
		with an arbitrary Borel measurable function $f \ge 0$ in place of $\G\mu$.


\begin{theorem}\label{main-thm-f}
		  Suppose $G$ is a  quasi-metric kernel in $\Omega\times \Omega$. 
		  Suppose  $0<q<1$, $\sigma \in \mathcal{M}^{+}(\Omega)$ ($\sigma\not=0$) and 
		  $f\ge 0$ is a Borel measurable function in $\Omega$. Then the following statements hold. 
		
		{(i)} Any nontrivial solution $u$ to equation \eqref{sublin-eq-f}  is unique and 
					satisfies the bilateral pointwise estimates  
				\begin{equation}
				\label{sublin-low-f} 
		 c \,   \left[ (\G \sigma)^{\frac{1}{1-q}} + 	 \K\sigma +  \G (f^q d \sigma)\right] +f \le 	u 
		\end{equation}
		and 
		\begin{equation}\label{sublin-up-f} 
		u  \le  \,  C \, \left[ (\G \sigma)^{\frac{1}{1-q}} + 	 \K\sigma + \G (f^q d \sigma)\right]+f ,   
				\end{equation}	
		$d \sigma$-a.e. in $\Omega$, 
		where $c, C$ are positive constants which depend only on $q$ and the 
		quasi-metric constant $\kappa$ of the kernel $G$.

		{(ii)}  Estimate \eqref{sublin-low-f}  holds for any nontrivial supersolution $u>0$      
		 such that  
		\begin{equation}
			\label{sub-f} 
\G(u^q d \sigma) + f \le u <\infty\quad d \sigma\textrm{-a.e.} \, \, {\rm in} \, \, \Omega, 
\end{equation}		
		whereas  estimate 
		 \eqref{sublin-up-f}  holds for any subsolution $u$  such that 
		\begin{equation}
			\label{super-f} 
u \le \G(u^q d \sigma) + f <\infty \quad d \sigma\textrm{-a.e.} \, \, {\rm in} \, \, \Omega.
\end{equation}

		{(iii)} A nontrivial solution $u$ to \eqref{sublin-eq-f}   exists if and only if 
	\begin{equation}	 \label{exist-f}
	\G \sigma <\infty, \, \, \K \sigma <\infty, \, \,  f <\infty, \, \,  \G(f^q d \sigma) <\infty, \quad d \sigma\textrm{-a.e.} \, \, {\rm in} \, \, \Omega.
\end{equation}	 
	\end{theorem}
	
	\begin{proof} We first prove statement (ii). Suppose  that $u$ is a nontrivial supersolution to 
	\eqref{sublin-eq-f}.  Let $v\defeq \G(u^q d \sigma)$. Then obviously $v>0$ and 
	$v+f\le u<\infty$ $d \sigma$-a.e.  By Lemma \ref{soln_abs_cont}, $v+f \in L^q_{{\rm loc}}(\Omega)$, and in particular 
	$d \mu \defeq f^q d \sigma$ is a measure in $\M(\Omega)$.   	
Since $\G [(v+f)^q d \sigma] \le v$ and $2^{q-1} (v^q+f^q)\le (v+f)^q$, we estimate  
	\begin{equation}\label{lower-f-a}
	2^{q-1}\left[  \G(v^q d \sigma) + \G \mu\right] \le v<\infty  \quad d \sigma\textrm{-a.e.}
	\end{equation}	
		The constant $2^{q-1}$ is easily incorporated into $\sigma$ 
		by using $\tilde \sigma\defeq 2^{q-1} \sigma$. Hence, 
			by Lemma \ref{qmm-wmp} and Corollary \ref{cor-lower-est}  with $\tilde\sigma$ in place of $\sigma$, 
			we deduce  
		\begin{equation}\label{lower-f}
	v \ge c \, \left[(\G \sigma)^{\frac{1}{1-q}} + \K \sigma + \G \mu \right] \quad d \sigma\textrm{-a.e.},  
				\end{equation}	
				where $c=c (q, \kappa)$. 
				 Since $u\ge v+f$,   we obtain \eqref{sublin-low-f}.

			To prove the upper  bound \eqref{sublin-up-f}, suppose that $u\ge 0$ is a subsolution to \eqref{sublin-eq-f}. Without 
				loss of generality we may assume that 
				\[
				(\G \sigma)^{\frac{1}{1-q}} + 	 \K\sigma + \G (f^q d \sigma)\not\equiv \infty, 
				\]
				since otherwise \eqref{sublin-up-f}  is trivial. From this it follows, as was shown above,  that $\int_{B(x, r)} f^q \, d \sigma< \infty$ 
				for all $x \in \Omega$, $r>0$, and consequently $d \mu\defeq f^q  d \sigma$ is a measure in $\M(\Omega)$. 
				By Lemma \ref{equiv-lemma},  
				\begin{equation}\label{d_sigma-finite}
				(\G \sigma)^{\frac{1}{1-q}} + 	 \K\sigma + \G \mu< \infty \quad 
				d \sigma\textrm{-a.e.}  
					\end{equation}

Let 
$v \defeq\G (u^q d\sigma)$. Then $u \le v +f<\infty$ $d \sigma$-a.e. by \eqref{super-f}. 
Hence,  
		\begin{align*}		
				 v & \le  \G \Big( (v +f)^q d\sigma \Big) 
				 \le \G \left (v^q d\sigma \right) + \G(f^q  d \sigma) \\ &
				 = \G  (v^q d \sigma) +\G \mu \quad d \sigma\textrm{-a.e.} 
				 	\end{align*}
					
We next show that $\G  (v^q d \sigma)<\infty$ $d \sigma$-a.e. Using Lemma \ref{qm_lemma-upper} 
with $d \nu \defeq u^q d\sigma$, so that $v=\G \nu$, 
we estimate 
	\begin{align*}	
	\G  (v^q d \sigma) = \G [(\G \nu)^q d \sigma] & \le (2 \kappa)^q \, (\G \nu)^q \,
	 [\G \sigma + (\K \sigma)^{1-q}] \\& = (2 \kappa)^q \, v^q \,
	 [\G \sigma + (\K \sigma)^{1-q}] <\infty \quad d \sigma\textrm{-a.e.}  
	 	\end{align*}	
	 by \eqref{super-f} and \eqref{d_sigma-finite}. It follows that $v$ is a subsolution satisfying 
	 \eqref{super-f} with $f=\G \mu$, i.e., 
\[
v \le \G(v^q d \sigma) + \G \mu< \infty \quad d \sigma\textrm{-a.e.}  
\]
Hence, by Corollary \ref{qm_cor-upper-K}, 
\[
v \le (8 \kappa)^{\frac{q}{1-q}} \, \left[(\G \sigma)^{\frac{1}{1-q}} + \K \sigma + \G \mu\right] \quad d \sigma\textrm{-a.e.}  
\] 
Thus,
\[
u \le v + f\le (8 \kappa)^{\frac{q}{1-q}} \, \left[(\G \sigma)^{\frac{1}{1-q}} + \K \sigma + \G \mu\right] +f. 
\]
This completes the proof of statement (ii).

We next  prove statement (i). Notice that the bilateral  estimates 
\eqref{sublin-low-f} and \eqref{sublin-up-f} 
for any solution $u$ to \eqref{sublin-eq-f} follow from statement 
(ii). It remains to prove the uniqueness property, which  is a consequence of these estimates. Indeed, suppose $u_1$ and $u_2$ are solutions 
to \eqref{sublin-eq-f}. Applying  \eqref{sublin-low-f} and \eqref{sublin-up-f} with $u_1$ and $u_2$ 
in place of $u$, we deduce 
 \[
 a \, u_1 \le u_2 \le   a^{-1} \, u_1 \qquad d\sigma\textrm{-a.e.},
 \]
where $a=c\, C^{-1}\le 1$ is a positive constant. Raising to the power $q$ and applying the operator $\G^\sigma$, we deduce 
 \[
 a^{q} \, \G(u^q_1 d \sigma) \le \G(u^q_2 d \sigma)\le a^{-q} \, \G(u^q_1 d \sigma).
 \]
It follows that 
 \[
 a^{q} \, [\G(u^q_1 d \sigma) +f] \le \G(u^q_2 d \sigma) +f\le a^{-q} \, [\G(u^q_1 d \sigma) +f].
 \]
Hence, 
 \[
 a^{q} \, u_1 \le u_2 \le a^{-q} \, u_1  \qquad d\sigma\textrm{-a.e.}  
 \]
Iterating this procedure, we obtain, for any $j=1, 2, \ldots$, 
 \[
 a^{q^j} \, u_1 \le u_2 \le a^{-q^j} \, u_1  \qquad d\sigma\textrm{-a.e.}
 \]
Passing to the limit as $j \to \infty$ yields $u_1=u_2$ $d\sigma$-a.e. This completes 
the proof of statement (i).

The proof of statement (iii) is similar to that of Lemma \ref{exist-lemma} in the special case 
$f= \G \mu$; we omit the details. 
\end{proof}

\noindent
{\bf Remark.} An analogue of  Theorem \ref{main-thm-f}  for quasi-metrically modifiable kernels $G$  
 is deduced in a similar way. This gives an extension of Theorem \ref{strong-thm-mod} for solutions of equation  \eqref{sublin-eq-f},  with $f$ in place of $\G\mu$. The corresponding estimates of solutions   remain valid once we replace $\G\mu$ with $c \, \G(f^q d \sigma)+f$ in \eqref{sublin-low-mod}, and $C \, \G(f^q d \sigma)+f$ in \eqref{sublin-up-mod}, respectively.


\end{document}